\documentclass{article}

\usepackage[T1]{fontenc}
\usepackage[utf8]{inputenc}
\usepackage[margin=1.2in]{geometry}
\usepackage{graphicx}
\usepackage{kantlipsum}

\usepackage{amsmath,amssymb,amsthm}
\numberwithin{equation}{section}

\usepackage[
natbib,
maxcitenames=3, 
maxbibnames=10,  
sorting=nyt,
url=false,
doi=false,
sortcites,
defernumbers,
backend=biber
]{biblatex}
\usepackage{hyperref}
\usepackage{ifthen}

\usepackage{todonotes}
\usepackage{url}

\usepackage[nameinlink, capitalise, noabbrev]{cleveref}
\crefname{algocf}{Algorithm}{Algorithms}
\usepackage{algorithm}
\usepackage{algpseudocode}

\usepackage{tabularx}
\usepackage{caption}
\usepackage{subcaption}
\usepackage{xfrac}
\usepackage{nicefrac}
\usepackage{tikz}
\usepackage{pgfplots}
\usetikzlibrary{3d,calc}

\usepackage{multirow}
\usepackage{hhline}
\usepackage{adjustbox}

\usepackage{csquotes}

\newtheorem{theorem}{Theorem}[section]
\newtheorem{proposition}[theorem]{Proposition}
\newtheorem{lemma}[theorem]{Lemma}

\theoremstyle{remark}

\theoremstyle{definition}
\newtheorem{example}[theorem]{Example}

\numberwithin{equation}{section}

\newcommand{\N}{\mathbb{N}}

\newcommand{\norm}[1]{\left\Vert #1 \right\Vert}
\newcommand{\abs}[1]{\left\vert #1 \right\vert}
\DeclareMathOperator{\divtmp}{div}
\renewcommand{\div}{\divtmp}
\DeclareMathOperator{\argmin}{arg\,min}
\DeclareMathOperator{\argmax}{arg\,max}

\newcommand{\st}{\,:\,}

\renewcommand{\d}{\,\mathrm{d}}

\newcommand{\eps}{\varepsilon}

\renewcommand{\L}{\mathsf{L}}

\newcommand{\M}{\mathcal M}
\newcommand{\grad}{\nabla}

\DeclareMathOperator{\kernel}{\mathsf{k}}
\DeclareMathOperator{\ccenter}{\mathsf{c}}
\newcommand{\mean}{\mathsf{m}}
\newcommand{\wmean}{\mathsf{m}_{\beta}}
\newcommand{\wcov}{\mathsf{C}_{\beta}}
\newcommand{\wkmean}{\mathsf{m}_{\beta,\kernel}}

\newcommand{\wkcov}{\mathsf{C}_{\beta,\kernel}}
\newcommand{\nR}{\mathbb{R}}

\newcommand{\nc}{\normalcolor}
\newcommand{\nvar}{\mathsf{V}}
\DeclareMathOperator{\tr}{trace}

\definecolor{tangelo}{rgb}{0.98, 0.3, 0.0}
\newcommand{\rev}{}

\makeatletter
\let\blx@rerun@biber\relax
\makeatother

\pgfplotsset{compat=1.17}
\usepackage{adjustbox}
\usepackage{tabularx}
\usepackage{xcolor}

\begin{document}

\title{Polarized consensus-based dynamics\\for optimization and sampling}
\author{
Leon Bungert
\thanks{Institute of Mathematics, University of Würzburg, Emil-Fischer-Str. 40, 97074 Würzburg, Germany. \href{mailto:leon.bungert@uni-wuerzburg.de}{leon.bungert@uni-wuerzburg.de}} 
\and 
Tim Roith\thanks{Deutsches Elektronen-Synchrotron, Hamburg,
Notkestraße 85, 22607 Hamburg, Germany. \href{mailto:tim.roith@desy.de}{tim.roith@desy.de}}
\and 
Philipp Wacker\thanks{School of Mathematics and Statistics, University of Canterbury, Science Road, Ilam, Christchurch 8140, NZ. \href{mailto:philipp.wacker@canterbury.ac.nz}{philipp.wacker@canterbury.ac.nz}}}

\maketitle

\begin{abstract}
In this paper we propose polarized consensus-based dynamics in order to make consensus-based optimization (CBO) and sampling (CBS) applicable for objective functions with several global minima or distributions with many modes, respectively. 
For this, we ``polarize'' the dynamics with a localizing kernel and the resulting model can be viewed as a bounded confidence model for opinion formation in the presence of common objective.
Instead of being attracted to a common weighted mean as in the original consensus-based methods, which prevents the detection of more than one minimum or mode, in our method every particle is attracted to a weighted mean which gives more weight to nearby particles.  
\rev We prove that in the mean-field regime the polarized CBS dynamics are unbiased for Gaussian targets.
We also prove that in the zero temperature limit and for sufficiently well-behaved strongly convex objectives the solution of the Fokker--Planck equation converges in the Wasserstein-2 distance to a Dirac measure at the minimizer.
Finally, \nc we propose a computationally more efficient generalization which works with a predefined number of clusters and improves upon our polarized baseline method for high-dimensional optimization.
\end{abstract}


\def\holzel{true}
\ifthenelse{\boolean{\holzel}}{%
\vspace{0.5em}

\begin{minipage}{\textwidth}
\itshape
\hfill
\begin{minipage}{.5\textwidth}
\enquote{%
Ich überleg' bei mir\\
Die nächsten Means dafür\\
Währenddessen ich noch rausch'.\\
Die Clusterzentren sind mir wohlbekannt,\\
Ich mein', wir ham an Konsens auch.}
\end{minipage}
\end{minipage}
\vspace{.5em}

\begin{minipage}{.95\textwidth}
\hfill loosely based on Hans Hölzel.
\end{minipage}}

\section{Introduction}

Partially motivated by the success of machine learning methods, which involve the minimization of high-dimensional and strongly non-convex objectives, in recent years the interest in consensus-based methods which do not rely on first-order gradient information has constantly increased.
Typically, zero-order optimization methods either construct a surrogate of the gradient and then perform a gradient-descent-type update \cite{conn2009introduction} or use a swarm model \cite{kennedy1995particle,reynolds1987flocks} where particles are attracted to the particle in the swarm which has the lowest objective value.
Notably, the former approach can lead to problems for strongly non-convex objectives with many critical points.
This issue is circumvented by swarm models, however, they typically do not allow for a mean-field description which makes their mathematical understanding difficult.

In contrast, consensus-based methods aim to achieve consensus by letting particles $\{x^{(i)}\}_{i=1}^J$ explore the objective landscape while attracting them to the weighted average of their positions with respect to the Gibbs measure $\pi\propto\exp(-\beta V)$ of the objective function $V:\nR^d\to\nR$, where $\beta>0$ is an inverse heat parameter. 
This method was first introduced in \cite{pinnau2017consensus} as consensus-based optimization (CBO).
As opposed to most other swarm methods CBO has a mean-field formulation involving a nonlinear Fokker--Planck equation.
In \cite{carrillo2018analytical} convergence to consensus of this equation was first proved, and \cite{ha2020convergence} showed consensus formation of the discrete particle method.
The key property for showing that the consensus is achieved close to the global minimum of the objective $V$ is that $\exp(-\beta V)$ concentrates around the global minimizer of $V$ as $\beta\to\infty$. 
More recently, \cite{fornasier2021consensus_conv} presented an improved convergence analysis which weakens some of the assumptions in \cite{carrillo2018analytical} and directly proves convergence to a point close to the global minimum in a Wasserstein-2 distance.

\begin{figure}[t!]
\parbox{3mm}{\rotatebox[origin=c]{90}{Standard CBO}}
\begin{subfigure}{.31\textwidth}
\includegraphics[width=\textwidth]{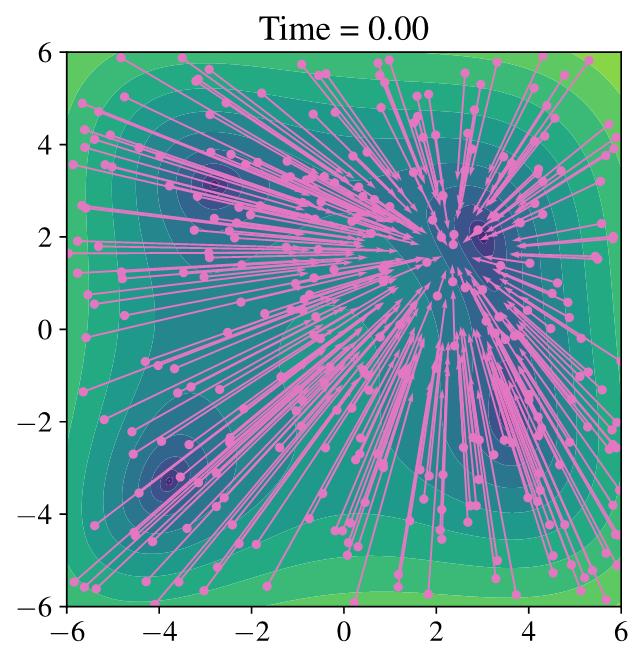}
\end{subfigure}%
\hfill%
\begin{subfigure}{.31\textwidth}
\includegraphics[width=\textwidth]{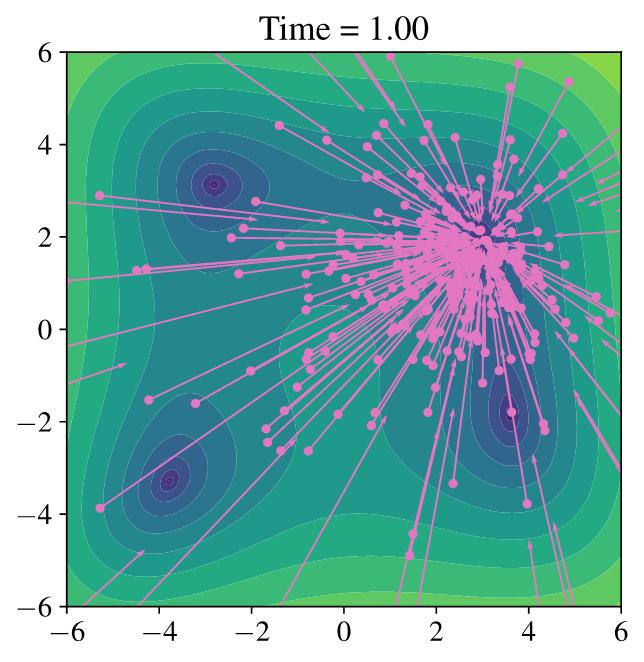}
\end{subfigure}%
\hfill%
\begin{subfigure}{.31\textwidth}
\includegraphics[width=\textwidth]{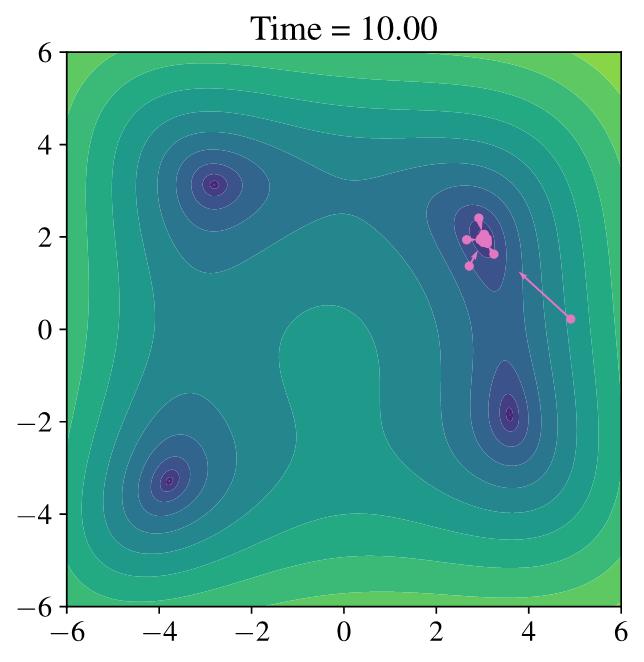}
\end{subfigure}%
\\
\parbox{3mm}{\rotatebox[origin=c]{90}{Polarized CBO}}
\begin{subfigure}{.31\textwidth}
\includegraphics[width=\textwidth]{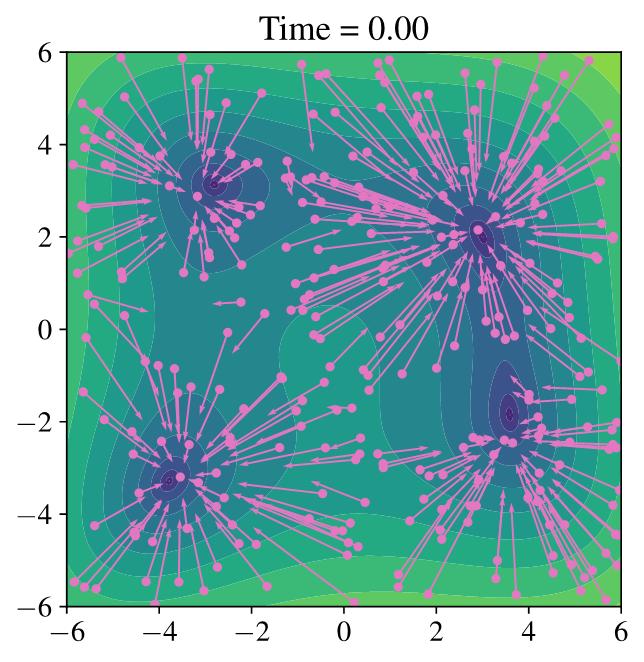}
\end{subfigure}%
\hfill%
\begin{subfigure}{.31\textwidth}
\includegraphics[width=\textwidth]{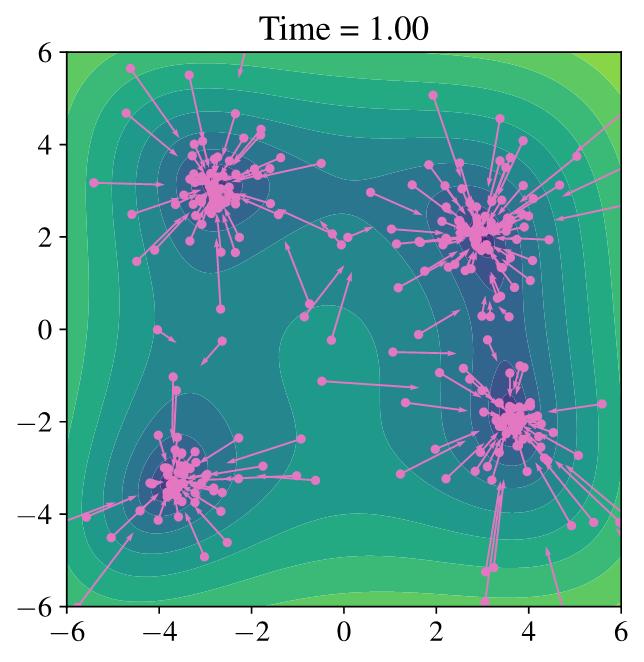}
\end{subfigure}%
\hfill%
\begin{subfigure}{.31\textwidth}
\includegraphics[width=\textwidth]{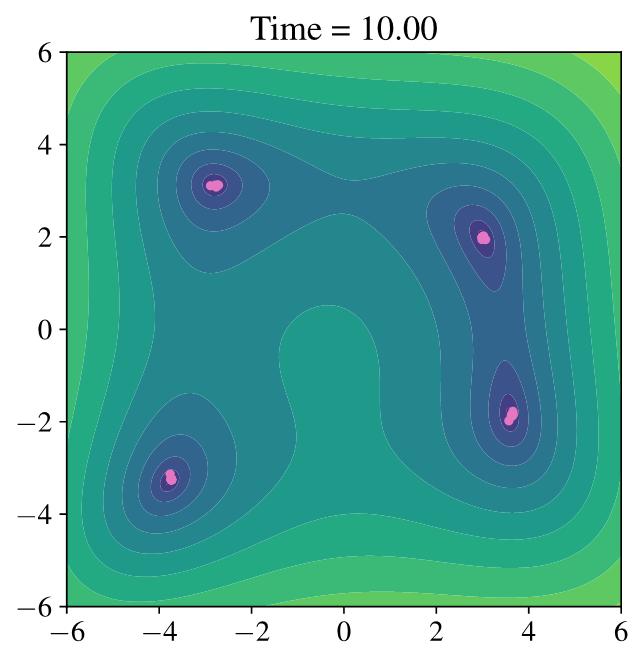}
\end{subfigure}%
\caption{Dynamics of standard and the proposed polarized CBO for mininizing the Himmelblau function.
The points mark particle locations, the arrows the drift field towards the weighted means.\label{fig:Himmelblau_KCBO}}
\end{figure}

Following up on the original formulation of CBO, numerous extensions were suggested.
In \cite{totzeck2020personal} an additional drift term, modelling the time-average of the personal best of every particle, is proposed.
In \cite{carrillo2021consensus_constrained} an extension of the CBO method for constrained problems is suggested, and \cite{carrillo2021consensus_ML} adapted the method for high-dimensional problems from machine learning by introducing random batching and changing the noise model.
In a different line of work, CBO methods on hypersurfaces were studied in \cite{fornasier2020consensus_hyper} and applications to machine learning were investigated in \cite{fornasier2021consensus_ML}.
In \cite{schillings2022ensemble} CBO was enriched by ensemble-based gradient information which comes at low computational cost and can improve upon the baseline method.
For an overview of recent developments we also refer to the review \cite{totzeck2022trends}.
Consensus-based methods have also been transferred to sampling. 
The work \cite{carrillo2022consensus_sampling} proposes a consensus-based sampling (CBS) method by changing the noise term in CBO to include a weighted sample covariance.
This prevents a collapse of the ensemble to full consensus and under suitable assumptions the method was shown to converge to a Gaussian approximation of the distribution $\exp(-V)$. 

While existing consensus-based methods have proven \rev to\nc\ work very well for non-convex objectives with many spurious local extreme points, they all suffer from the conceptual drawback that by definition they can at most \rev approximate\nc\ one minimum, respectively one mode in the context of sampling.
Therefore, the goal of this paper is to design consensus-based particle dynamics which support multiple consensus points, or in other words, polarization.

The central idea for the method which we are presenting in this paper is based on the following though experiment: 
Assuming that two clusters of particles have formed, each one centered around a global minimum, we do not want to compute a shared weighted mean of their positions. 
This would pull one of the clusters into the other one. 
Therefore, we replace the weighted mean---being an integral component of current consensus-based methods---by a collection of localized means which additionally \rev weight \nc particle positions by their proximity to the considered particle.
The localization is achieved by a suitable kernel function.
This then leads to polarized dynamics where multiple consensus points can be reached which \rev opens up the possibility of finding multiple global minima or multiple modes, respectively. Standard CBO, on the other hand is bound to converge to a single minimum.
\nc

As we shall discuss in more detail later, our approach carries strong similarities to bounded confidence models of opinion formation, introduced in \cite{deffuant2000mixing} but see also \cite{hegselmann2002opinion,gomez2012bounded,fortunato2005vector}.
There agents only interact with each other if their opinions are sufficiently close, which in the end can lead to formation of multiple consensus points and is referred to as polarization of opinions in \cite{hegselmann2002opinion}.

\paragraph{Main contributions}

The focus of this paper is the development of the novel polarized consensus-based dynamics and an extensive numerical evaluation of the method. 
The mathematical models we propose come with a lot of new theoretical questions, as well, which will be the subject of future investigations. 
The main contributions and structure of this paper can be summarized as follows:

\begin{itemize}
    \item \cref{ss:polarized_CBO}: We propose a novel polarized computation of weighted means for CBO methods.
    \item \cref{ss:cluster_CBO}: We propose an algorithmic variant which uses a predetermined number of cluster points to compute weighted means.
    \item \cref{ss:polarized_CBS}: We propose a novel polarized computation of weighted covariances for the CBS method and prove that it is unbiased for Gaussian targets.
    \item \rev \cref{sec:analysis}: We prove convergence of the mean-field dynamics in the Wasserstein-2 distance for sufficiently well-behaved objective functions (\cref{thm:convergence_W_2}). \nc
    \item  \cref{sec:numerics}: We conduct extensive numerical and statistical evaluations of our polarized optimization method, showing that it can find multiple global minima in low and high dimensional optimization problems and can even improve upon standard CBO for the detection of one minimum.
    We also test our polarized CBS method for sampling from mixtures of Gaussian and a non-Gaussian distribution where it exhibits better performance than standard CBS.
\end{itemize}

\section{Models}

In this section we describe in detail how standard consensus-based models work and how we generalize them with our polarized approach.
For this, we let $V:\nR^d\to[0,\infty)$ be a (possibly non-convex) objective function and $\beta>0$ an inverse heat parameter.

\subsection{Polarized consensus-based optimization}\label{ss:polarized_CBO}

Given a measure $\rho\in\M(\nR^d)$, for standard CBO one defines a weighted mean as
\begin{align}
    \label{eq:weighted_mean}
    \wmean[\rho] 
    &:= 
    \frac{\int y \exp(-\beta V(y))\d\rho(y)}{\int \exp(-\beta V(y))\d\rho(y)}.
\end{align}
Here and in the rest of the paper all integrals will be over $\nR^d$.
The consensus-based optimization method from \cite{pinnau2017consensus} then amounts to solving the following system of stochastic differential equations (SDEs) for particles $\{x^{(i)}\}_{i=1,\dots,J}$:
\begin{align}\label{eq:CBO}
    \d x^{(i)} = -(x^{(i)} - \wmean[\rho])\d t + \sigma\abs{x^{(i)} - \wmean[\rho]}\d W^{(i)}, \qquad
    \rho := \frac{1}{J}\sum_{i=1}^J\delta_{x^{(i)}},
\end{align}
where $\{W^{(i)}\}_{i=1}^J$ denote independent Brownian motions and $\sigma\geq 0$ determines the strength of the randomness in the model.
The Fokker--Planck equation associated to \labelcref{eq:CBO} is the following PDE
\begin{align}\label{eq:fuck_kerplunk_CBO}
    \partial_t \rho_t(x) = \div\Big(\rho_t(x)(x-\wmean[\rho_t])\Big) + \frac{\sigma^2}{2}\Delta\left(\rho_t(x)\abs{x - \wmean[\rho]}^2\right).
\end{align}
As explained above, this dynamical system forces particles to collapse to consensus, meaning that under certain conditions on $V$ the empirical measures $\rho(t)$ converge to $\delta_{\tilde x}$ as $t\to\infty$, where for $\beta\to\infty$ the consensus-point $\hat{x}$ converges to the global minimizer of $V$, see \cite{carrillo2018analytical,fornasier2021consensus_conv}.

We now explain our polarized modification of CBO for optimizing objective functions with multiple global minima.
Given a measure $\rho\in\M(\nR^d)$ and a kernel function $\kernel:\nR^d\times\nR^d\to[0,\infty)$ we define the weighted mean
\begin{align}
    \label{eq:wk_mean}
    {%
    \wkmean[\rho](x)
    := 
    \frac{\int \kernel(x,y) y \exp(-\beta V(y))\d\rho(y)}{\int \kernel(x,y) \exp(-\beta V(y))\d\rho(y)}, \quad x\in\nR^d.}
\end{align}
The corresponding polarized optimization dynamics take the form
\begin{align}\label{eq:our_CBO}%
\boxed{%
\d x^{(i)} = -(x^{(i)} - \wkmean[\rho](x^{(i)}))\d t + \sigma\abs{x^{(i)} - \wkmean[\rho](x^{(i)})}\d W^{(i)},}
\end{align}
where $\rho := \frac{1}{J}\sum_{i=1}^J\delta_{x^{(i)}}$. The Fokker--Planck equation associated to \labelcref{eq:our_CBO} is the following PDE
\begin{align}\label{eq:polarized_fuck_kerplunk_CBO}
\resizebox{.9\textwidth}{!}{$%
    \partial_t \rho_t(x) = \div\Big(\rho_t(x)(x-\wkmean[\rho_t](x))\Big) + \frac{\sigma^2}{2}\Delta\left(\rho_t(x)\abs{x - \wkmean[\rho_t](x)}^2\right).$}
\end{align}
Note that the idea of letting the dynamics of particle $x^{(i)}$ mainly depend on spatially close particles, as modelled through the kernel $\kernel$, has strong similarities to bounded confidence models of opinion dynamics introduced in \cite{deffuant2000mixing}.
In these models, typically there is no objective function to be minimized and the kernel is of the form $\kernel(x,y) := 1_{\abs{x-y}\leq\kappa}(x,y)$, where $\kappa>0$ is a so-called confidence level.
The simplest such dynamics then take the form
\begin{align}\label{eq:bounded_confidence}
    \frac{\d x^{(i)}}{\d t} = - \left(x^{(i)} - \frac{1}{N(x^{(i)})}\sum_{j=1}^J 1_{\abs{x^{(i)}-x^{(j)}}\leq\kappa}x^{(j)}\right),
\end{align}
where $N(x^{(i)}) := \#\{1\leq j \leq J \st \abs{x^{(i)}-x^{(j)}}\leq\kappa\}$ denotes the numbers of points which are not farther than $\kappa$ away from $x^{(i)}$.
Notably, \labelcref{eq:bounded_confidence} coincides with \labelcref{eq:our_CBO} for the special case of $\kernel(x,y) := 1_{\abs{x-y}\leq\kappa}(x,y)$, a constant objective $V\equiv const$, and $\sigma=0$.
More generally, dynamics of the form \labelcref{eq:our_CBO,eq:bounded_confidence} can be viewed as processes on co-evolving networks or graphs where the weights between different particles $x^{(i)}$ and $x^{(j)}$ depend on the kernel $\kernel$ and the loss function $V$. 
We refer to \cite{burger2022kinetic,burger2021network} for a unified description and the study of mean-field equations for general processes of this form.
\rev
Furthermore, for $V=const$ and $\sigma=0$, a time discretization of \cref{eq:our_CBO} via an Euler--Maruyama scheme with stepsize $\d t =1$, yields the update
\begin{align*}
x^{(i)} \gets x^{(i)} -(x^{(i)} - \wkmean[\rho](x^{(i)})) = 
\frac{\sum_{j=1}^N \kernel(x^{(i)},x^{(j)})\, x^{(j)}}{\sum_{j=1}^N
\kernel(x^{(i)},x^{(j)})},
\end{align*}
which is known as the \emph{mean-shift} algorithm \cite{schnell1964methode,fukunaga1975estimation}.
\nc
In the following, we would like to discuss two important special cases of our model:
If one chooses the kernel $\kernel(x,y)=1$ for all $x,y\in\nR^d$, then \labelcref{eq:wk_mean,eq:our_CBO,eq:polarized_fuck_kerplunk_CBO} simply reduce to the standard CBO setup \labelcref{eq:weighted_mean,eq:CBO,eq:fuck_kerplunk_CBO}.
Hence, our method is a generalization of CBO.
On the other hand, if one chooses the Gaussian kernel $\kernel(x,y) := \exp\left(-\frac{\abs{x-y}^2}{2\kappa^2}\right)$, then the weighted mean $\wkmean[\rho](x)$ can be rewritten as
\begin{align}
    \wkmean[\rho](x) := \frac{\int y \exp\left(-\beta V(y) - \frac{\abs{x-y}^2}{2\kappa^2}\right)\d\rho(y)}{\int \exp\left(-\beta V(y) - \frac{\abs{x-y}^2}{2\kappa^2}\right)\d\rho(y)}.
\end{align}
In this case our method can be regarded as standard CBO applied to a spatially varying quadratic regularization of the objective function $y\mapsto V_x(y)$, defined as
\begin{align*}
    V_x(y) := V(y) + \frac{1}{2\kappa^2\beta}\abs{x-y}^2.
\end{align*}
Note that the central difference between the standard and our method is that the weighted mean \labelcref{eq:wk_mean} depends on the particle position $x$ and is not same for all particles. Especially for high-dimensional problems it was shown in \cite{carrillo2021consensus_ML} that the performance of CBO can be significantly improved when using a coordinate-wise noise model. 
To be precise they suggest the following replacement in \labelcref{eq:our_CBO}
\begin{align}\label{eq:compNoise}
    \abs{x^{(i)} - \wkmean[\rho](x^{(i)})}\d W^{(i)}
    \longrightarrow
    \sum_{n=1}^d \vec e_n \left(x^{(i)} - \wkmean[\rho](x^{(i)})\right)_n \d W^{(i)}_n,
\end{align}
where $\vec e_n$ denotes the $n$-th unit vector in $\nR^d$.
This changes the Laplacian term in the corresponding Fokker--Planck equation \labelcref{eq:polarized_fuck_kerplunk_CBO} to
\begin{align*}
    \Delta(\rho_t(x)\abs{x-\wkmean[\rho_t](x)}
    \longrightarrow
    \sum_{n=1}^d \partial_{nn}^2\Big(\rho_t(x)(x - \wkmean[\rho_t](x))_n\Big).
\end{align*}
For computing an approximate solution of the stochastic CBO dynamics with either of the two discussed noise models, we employ a standard Euler--Maruyama scheme which we sum up in \cref{alg:CBO}.
There the function \textbf{ComputeMean} determines the precise variant of CBO that is being used. 
In \cref{alg:PolarCBO} we specify this function for the proposed polarized CBO which reduces to standard CBO for a constant kernel $\kernel\equiv 1$.
In the next section we will discus an algorithmic variant (see \cref{alg:CCBO}) of the mean computation \labelcref{eq:wk_mean} which is computationally more efficient and exhibits empirical advantages in high dimensions.

\begin{algorithm}
\caption{General consensus-based optimization}\label{alg:CBO}
\textbf{Input:} Time step size $\d t>0$, initial particles $\{x^{(j)}\}_{j=1}^J$, diffusion parameter $\sigma\geq 0$, noise~model~$\mathsf{M}$
\begin{algorithmic}
\For{$i=1,\dots,J$}
    \State $\mean^{(i)} \gets \textbf{ComputeMean}$
    \If{$\mathsf{M}$ = Standard noise}
        \State $\xi \sim \mathcal{N}(0,\d t \cdot I_{d\times d})$
        \State $n^{(i)} \gets \abs{x^{(i)} - \mean^{(i)}}\xi$
    \EndIf
    \If{$\mathsf{M}$ = Coordinate noise}
        \State $\xi \sim \mathcal{N}(0,\d t\cdot I_{d\times d})$
        \State $n^{(i)} \gets \sum_{n=1}^d e_n (x^{(i)} - \mean^{(i)})_n \xi_n$
    \EndIf
    \State $x^{(i)} \gets x^{(i)} - \d t (x^{(i)} - \mean^{(i)}) + \sigma \,n^{(i)}$
\EndFor
\end{algorithmic}
\end{algorithm}

\begin{algorithm}
\caption{\textbf{ComputeMean} for polarized CBO}\label{alg:PolarCBO}
\textbf{Input:} particle $x^{(i)}$, particle swarm $\{x^{(j)}\}_{j=1}^J$, kernel $\kernel$, objective $V$, inverse temperature $\beta>0$\\
\textbf{Output:} $\mean^{(i)} \gets \frac{\sum_{j=1}^J x^{(j)} \kernel(x^{(i)},x^{(j)})\exp(-\beta V(x^{(j)}))}{\sum_{j=1}^J\kernel(x^{(i)},x^{(j)})\exp(-\beta V(x^{(j)}))}$
\end{algorithm}

\subsection{Cluster-based model}
\label{ss:cluster_CBO}

In this section we propose an algorithmic alternative to the weighted mean, defined in \labelcref{eq:wk_mean} with the motivation of making the computation of the weighted means computationally more efficient and to also encourage polarizing effects between the different means.
Given particles $x^{(i)}$ for $i=1,\ldots,J$, standard CBO uses one weighted mean whereas our polarized version uses $J$ many weighted means. 

As an alternative model we consider cluster means $\ccenter^{(j)}$ for $j=1,\ldots, J_c$, where $J_c\leq J$. 
We encode the probability that the particle $x^{(i)}$ belongs to the cluster mean $\ccenter^{(j)}$ with $p_{ij}>0$.
\rev Given initial probabilities $p_{ij}>0$, we perform the following iterative\nc\ update for each $j=1,\ldots, J_c$,
\begin{subequations}\label{eq:prob_update}
\begin{alignat}{2}
p_{i}^{\text{max}} &:= \max_{j =1,\ldots,J_c} p_{ij}\quad&&\text{ for }i=1,\ldots, J,\\
r_{ij}&:= \left(\frac{p_{ij}}{p_{i}^{\text{max}}}\right)^\alpha
\quad&&\text{ for }i=1,\ldots, J,\\
\tilde p_{ij} &:= r_{ij} \kernel(x^{(i)}, \ccenter^{(j)}).&&
\end{alignat}
\end{subequations}
\rev
The new values of the probabilities $p_{ij}$ are then obtained by renormalization over $j$, i.e.,
\begin{align}\label{eq:normalization}
p_{ij} \gets \frac{\tilde{p}_{ij}}{\sum_{j=1}^{J_c} \tilde{p}_{ij}}
\end{align}
and the clusters means are then updated via
\begin{align}\label{eq:ccenter_update}
\ccenter^{(j)} \gets \frac{\sum_{i=1}^J x^{(i)} p_{ij} \exp(-\beta V(x^{(i)}))}{\sum_{i=1}^J p_{ij} \exp(-\beta V(x^{(i)}))}.
\end{align}
\nc\
Here, $\alpha\geq 0$ is a discounting coefficient.
The interpretation of the above scheme is as follows:
If the particle $x^{(i)}$ identifies the cluster center $\ccenter^{(j)}$ as ``belonging to it the most'', then $r_{ij}=1$ and the probability that $x^{(i)}$ indeed belongs to $\ccenter^{(j)}$ is only determined by the spatial proximity, encoded through $\kernel(x^{(i)},\ccenter^{(j)})$.
If, however, it ``feels more dedicated'' to another cluster mean $\ccenter^{(k)}$ (different from $\ccenter^{(j)}$), then $p_{ij} < p_{ik}$ and thus $r_{ij} < 1$. This results in a correction of the particle-cluster correspondence, \rev strengthening the bond to cluster $j$ of particle~$i$.\nc\
The exponent $\alpha\geq 0$ determines the strength of this additional polarization incentive, with larger $\alpha\geq 0$ leading to greater polarization.
In order to obtain the indivdual mean for each particle (like in the polarized CBO scheme) we simply compute
\begin{align}\label{eq:means}
\mean^{(i)} :=
\sum_{j=1}^{J_c} p_{ij} \ccenter^{(j)}.
\end{align}
We summarize the cluster-based mean computations in \cref{alg:CCBO}, which can be used in place of \cref{alg:PolarCBO} in the CBO scheme \cref{alg:CBO}.
\begin{algorithm}
\caption{\textbf{ComputeMean} for cluster-based method}\label{alg:CCBO}
\textbf{Input:} particle $x^{(i)}$, particle swarm $\{x^{(j)}\}_{j=1}^J$, cluster centers $\{\ccenter^{(j)}\}_{j=1}^{J_c}$, probabilities $\{p_{ij}\}_{j=1}^{J_c}$, discounting coefficient $\alpha\geq 0$, kernel $\kernel$, objective $V$, inverse temperature $\beta>0$
\begin{algorithmic}
\State $p_{i}^{\text{max}} \gets \max_{j =1,\ldots,N_c} p_{ij}$\Comment{Compute probability of likeliest cluster}
\For{$j=1,\dots, J_c$}
    \State $r_{ij}\gets \left(\frac{p_{ij}}{p_{i}^{\text{max}}}\right)^\alpha$
    \State $\tilde p_{ij} \gets r_{ij}\kernel(x^{(i)}, \ccenter^{(j)})$\Comment{Discount probabilities for other clusters}
\EndFor
\For{$j=1,\dots, J_c$}
    \State $p_{ij} \gets \frac{\tilde{p}_{ij}}{\sum_{j=1}^{J_c} \tilde{p}_{ij}}$\Comment{Normalize probabilities}
\EndFor
\For{$j=1,\dots, J_c$}
    \State $\ccenter^{(j)} \gets \frac{\sum_{i=1}^J x^{(i)} p_{ij} \exp(-\beta V(x^{(i)}))}{\sum_{i=1}^J p_{ij} \exp(-\beta V(x^{(i)}))}$\Comment{Update cluster centers}
\EndFor
\end{algorithmic}
\textbf{Output:}
$m_i \gets
\sum_{j=1}^{J_c} p_{ij} \ccenter^{(j)}$
\end{algorithm}

Note that for $\alpha=\infty$ we have that
\begin{align*}
    r_{ij} 
    :=
    \begin{cases}
    1,\qquad &\text{if }j\in\argmax_{j=1,\dots,J_c}p_{ij},\\
    0,\qquad&\text{else},
    \end{cases}
\end{align*}
meaning that the only probability which survive are the one for the likeliest cluster centers.
Another interesting special case arises when choosing the Gaussian kernel $\kernel(x,y) := \exp\left(-\frac{\abs{x-y}^2}{2\kappa^2}\right)$.
In this case one gets that
\begin{align*}
p_{ij} 
&:= \frac{\exp\left(-\frac{1}{2\kappa^2}\abs{x^{(i)}-\ccenter^{(j)}}^2 + \log r_{ij}\right)}{\sum_{j=1}^{J_c} \exp\left(-\frac{1}{2\kappa^2}\abs{x^{(i)}-\ccenter^{(j)}}^2 + \log r_{ij}\right)}
\\
&\longrightarrow
\begin{cases}
1\quad&\text{if }\abs{x^{(i)}-\ccenter^{(j)}} \leq \abs{x^{(i')}-\ccenter^{(j')}}\,\forall i'=1,\dots,J,\,j'=1,\dots,J_c,\\
0\quad&\text{else},
\end{cases}
\end{align*}
as $\kappa\to 0$. So for very small values of $\kappa$ a hard assignment of the points $x^{(i)}$ to the clusters $\ccenter^{(j)}$ based on spatial proximity is performed which is reminiscent of the k-means algorithm.

It is important to initialize all quantities correctly in order to obtain a meaningful algorithm. 
If one naively initializes $p_{ij} := \frac{1}{J_c}$ for all $i=1,\dots,J$ and computes initial cluster centers via \labelcref{eq:ccenter_update}, then $\ccenter^{(j)} = \frac{\sum_{i=1}^J x^{(i)} \exp(-\beta V(x^{(i)}))}{\sum_{i=1}^J\exp(-\beta V(x^{(i)}))}$ equals the standard CBO weighted mean for all $j$.
Correspondingly, the probability updates \labelcref{eq:prob_update,eq:normalization} will leave $p_{ij}$ untouched. 
In this case the method reduces precisely to standard CBO.
Therefore, we initialize the probabilities randomly by drawing $\tilde p_{ij} \sim \operatorname{Unif}(0,1)$ and normalizing with \labelcref{eq:normalization}.
Then we compute initial cluster means using \labelcref{eq:ccenter_update}.

Let us also mention that the complexity of CBO where means are computed with \cref{alg:CCBO} is order $O(J\cdot J_c)$ which is significantly smaller than $O(J^2)$ which is the complexity of CBO with the baseline polarized mean computation from \cref{alg:PolarCBO}. This is due to the fact that the cluster-based method models consensus and polarization of individuals with respect to existing opinionated groups, with the cluster center as a surrogate for these groups, whereas the polarized CBO approach tracks all particles' interactions with each other.

Last, we demonstrate how to obtain a SDE and mean-field interpretation of \cref{alg:CCBO}, where for conciseness we restrict ourselves to the case $\alpha=0$.
In this case \labelcref{eq:normalization,eq:ccenter_update} reduce to
\begin{align*}
    p_{ij} \gets \frac{\kernel(x^{(i)},\ccenter^{(j)})}{\sum_{j=1}^{J_c}\kernel(x^{(i)},\ccenter^{(j)})}
    ,
    \qquad
    \ccenter^{(j)} \gets \frac{\sum_{i=1}^J x^{(i)} \kernel(x^{(i)},\ccenter^{(j)}) \exp(-\beta V(x^{(i)}))}{\sum_{i=1}^J \kernel(x^{(i)},\ccenter^{(j)}) \exp(-\beta V(x^{(i)}))}.
\end{align*}
This allows us to express the cluster-based mean computation with $\alpha=0$ as the coupled system:
\begin{subequations}
\begin{align}
    \d x^{(i)} &= -(x^{(i)} - \mean^{(i)})\d t + \sigma\abs{x^{(i)} - 
    \mean^{(i)}}\d W^{(i)},
    \\
    \mean^{(i)} &=
    \frac{\sum_{j=1}^{J_c}\kernel(x^{(i)},\ccenter^{(j)})\ccenter^{(j)}}{\sum_{j=1}^{J_c}\kernel(x^{(i)},\ccenter^{(j)})},
    \\
    \ccenter^{(j)}
    &=
    \frac{\int x \kernel(x,\ccenter^{(j)})\exp(-\beta V(x))\d\rho_t(x)}{\int \kernel(x,\ccenter^{(j)})\exp(-\beta V(x))\d\rho_t(x)},
    \\
    \rho_t &:= \frac{1}{N}\sum_{i=1}^J \delta_{x^{(i)}}.
\end{align}
\end{subequations}
Note that \cref{alg:CCBO} approximates the fixed point equation for $\ccenter^{(j)}$ with one iteration of the fixed point map.
The corresponding mean-field system is readily obtained as:
\begin{subequations}
\begin{align}
    \partial_t \rho_t(x) &= \div(\rho_t(x)(x - \mean[\rho_t](x))) + \frac{\sigma^2}{2}\Delta\left(\rho_t(x)\abs{x - 
    \mean[\rho_t](x)}^2\right),
    \\
    \mean[\rho_t](x) &=
    \frac{\sum_{j=1}^{J_c}\kernel(x,\ccenter^{(j)})\ccenter^{(j)}}{\sum_{j=1}^{J_c}\kernel(x,\ccenter^{(j)})},
    \\
    \ccenter^{(j)}
    &=
    \frac{\int x \kernel(x,\ccenter^{(j)})\exp(-\beta V(x))\d\rho_t(x)}{\int \kernel(x,\ccenter^{(j)})\exp(-\beta V(x))\d\rho_t(x)}.
\end{align}
\end{subequations}
We expect that similarly one can derive a SDE and mean-field interpretation of the model with $\alpha>0$ but we do not endeavour this here.

\subsection{Polarized consensus-based sampling}\label{ss:polarized_CBS}

The last model variant that we consider here is an application to sampling.
In \cite{carrillo2022consensus_sampling} a sampling version of CBO was proposed and termed consensus-based sampling (CBS).
Defining a weighted covariance matrix as
\begin{align}
    \label{eq:weighted_cov}
    \wcov[\rho] 
    &:=
    \frac{\int (y-\wmean[\rho])\otimes(y-\wmean[\rho]) \exp(-\beta V(y))\d\rho(y)}{\int \exp(-\beta V(y))\d\rho(y)}.
\end{align}
CBS aims to sample from the measure $\exp(-V)$ and by solving the following system of SDEs:
\begin{align}\label{eq:CBS}
    \d x^{(i)} = -(x^{(i)} - \wmean[\rho])\d t + \sqrt{2\lambda^{-1}\wcov[\rho]}\d W^{(i)}, \qquad
    \rho := \frac{1}{J}\sum_{i=1}^J\delta_{x^{(i)}}.
\end{align}
Here the parameter $\lambda$ interpolates between an optimization method ($\lambda=1$) and a sampling method ($\lambda=(1+\beta)^{-1}$).
For the latter scaling of $\lambda$ a collapse of $\rho$ is avoided and $\rho$ samples from $\exp(-V)$ if this measure is Gaussian.
The Fokker--Planck equation associated with \labelcref{eq:CBS} is given by
\begin{align}\label{eq:fuck_kerplunk_CBS}
    \partial_t \rho_t(x) = \div\Big(\rho_t(x)(x-\wmean[\rho_t])\Big)+ \lambda^{-1}\div\Big(\wcov[\rho_t]\grad \rho_t(x)\Big).
\end{align}
We can polarize CBS by using the mean from \labelcref{eq:wk_mean} to define a weighted variance, as follows:
\begin{align}
{%
\wkcov[\rho](x)
:=
\frac{\int
\kernel(x,y)(y-\wkmean[\rho](x))\otimes(y-\wkmean[\rho](x)) \exp(-\beta V(y))\d\rho(y)}{\int \kernel(x,y) \exp(-\beta V(y))\d\rho(y)},}
\end{align}
for $x\in\nR^d$. The corresponding CBS dynamics are then given by
\begin{align}\label{eq:our_CBS}
\boxed{%
\d x^{(i)} = -(x^{(i)} - \wkmean[\rho](x^{(i)}))\d t + \sqrt{2\lambda^{-1}\wkcov[\rho](x^{(i)})}\d W^{(i)},}
\end{align}
where $\rho := \frac{1}{J}\sum_{i=1}^J\delta_{x^{(i)}}$. The associated Fokker--Planck equation is
\begin{align}\label{eq:polarized_fuck_kerplunk_CBS}
    {%
    \partial_t \rho_t(x) = \div\big(\rho_t(x)(x-\wkmean[\rho_t](x))\big)+\lambda^{-1}\div\big(\wkcov[\rho_t](x)\grad \rho_t(x)\big).}
\end{align}
Perhaps a little unexpectedly we can prove that, just as the Fokker--Planck equation of CBS \labelcref{eq:fuck_kerplunk_CBS}, for Gaussian kernels of arbitrary width our version \labelcref{eq:polarized_fuck_kerplunk_CBS} leaves Gaussians measures invariant, which is an important consistency property.

\begin{proposition}\label[proposition]{prop:unbiased}
The polarized CBS dynamics in sampling mode with a Gaussian kernel leaves a Gaussian target measure invariant. More precisely, let \\ $\kernel(x,y) := \exp\left(-\frac{1}{2}(x - y)^T \Sigma_1^{-1}(x - y)\right)$ for a symmetric and positive definite covariance matrix $\Sigma_1\in\mathbb{R}^{d\times d}$, and let
\begin{align*}
    V(y) := \frac{1}{2}(x - m)^T\Sigma_2^{-1}(x - m)
\end{align*}
for some $m\in\nR^d$ and a symmetric and positive definite covariance matrix $\Sigma\in\nR^{d\times d}$.
Then $\rho^\star$ defined as
\begin{align*}
    \rho^\star(x) := \exp(-V(x))
\end{align*}
is a stationary solution of the Fokker--Planck equation \labelcref{eq:polarized_fuck_kerplunk_CBS} for any $\beta>0$ and for $\lambda=(1+\beta)^{-1}$.
\end{proposition}%
\begin{proof}
We use \rev the\nc\ following formula for the product of two Gaussians, which can be found, for instance, in \cite{petersen2008matrix}, to obtain
\begin{align*}
    \kernel(x,y)\exp(-(1+\beta)V(y)) 
    &= 
    \exp\left(-\frac{1}{2}(x - y)^T \Sigma_1^{-1}(x - y)\right)\\
    &\exp\left(-\frac{1+\beta}{2}(x - m)^T\Sigma_2^{-1}(x - m)\right)
    \\
    &=
    c_x
    \exp\left(-\frac{1}{2}(y - m_x)^T \Sigma_3 (y-m_x)\right),
\end{align*}
where $c_x>0$ is a normalization constant and
\begin{align*}
    \Sigma_3 &:= \left(\Sigma_1^{-1} + (1+ \beta)\Sigma_2^{-1}\right)^{-1},
    \\
    m_x &:= \Sigma_3\left(\Sigma_1^{-1}x + (1+\beta)\Sigma_2^{-1}m\right).
\end{align*}
Using this we obtain
\begin{align*}
    \int \kernel(x,y) \exp(-\beta V(y))\d\rho^\star(y) 
    &=
    \int \kernel(x,y) \exp(-(1+\beta)V(y))\d y
    \\
    &=
    c_x 
    \int 
    \exp\left(-\frac{1}{2}(y - m_x)^T \Sigma_3 (y-m_x)\right)
    \d y
    \\
    &=
    c_x (2\pi)^\frac{d}{2}\det(\Sigma_3)^\frac{1}{2}.
\end{align*}
Similarly, we obtain
\begin{align*}
    \int y \kernel(x,y) \exp(-V(y))\d\rho^\star(y) 
    &=
    c_x 
    \int 
    y
    \exp\left(-\frac{1}{2}(y - m_x)^T \Sigma_3 (y-m_x)\right)
    \d y
    \\
    &=
    c_x (2\pi)^\frac{d}{2}\det(\Sigma_3)^\frac{1}{2} m_x.
\end{align*}
Combining the two we obtain
\begin{align*}
    \wkmean[\rho^\star](x)
    &=
    m_x.
\end{align*}
Similarly, we can compute the weighted covariance as
\begin{gather*}
    \wkcov[\rho^\star](x)
    =
    \frac{\int (x-\wkmean[\rho^\star](x))\otimes (x-\wkmean[\rho^\star](x)) \kernel(x,y) \exp(-\beta V(y))\d\rho^\star(y) }{\int \kernel(x,y) \exp(-V(y))\d\rho^\star(y)}
    \\
    =
    \frac{c_x \int (x-m_x) \otimes (x-m_x) c_x
    \exp\left(-\frac{1}{2}(y - m_x)^T \Sigma_3 (y-m_x)\right) \d y }{c_x(2\pi)^\frac{d}{2}\det(\Sigma_3)^\frac{1}{2}}
    =
    \Sigma_3.
\end{gather*}
Hence, we get for $\lambda = (1+\beta)^{-1}$:
\begin{align*}
&\phantom{{}={}}
\rho^\star(x) (x - \wkmean[\rho^\star](x)) + \lambda^{-1}\wkcov[\rho^\star](x)\grad\rho^\star(x)
\\
&=
\rho^\star(x) (x-m_x) + (1+\beta) \Sigma_3 \grad\rho^\star(x)
\\
&=
\rho^\star(x) (x-m_x) - (1+\beta) \Sigma_3 \grad V(x)\rho^\star(x)
\\
&=
\rho^\star(x)
\left(
(x-m_x) - (1+\beta) \Sigma_3\Sigma_2^{-1}(x-m)
\right)\rev .\nc
\end{align*}
Now we use the definition of $m_x = \Sigma_3(\Sigma_1^{-1}x+(1+\beta)\Sigma_2^{-1}m)$ to obtain
\begin{align*}
    &\phantom{{}={}}
    x - m_x - (1+\beta) \Sigma_3\Sigma_2^{-1}(x-m)
    \\
    &=
    x - \Sigma_3(\Sigma_1^{-1}x+(1+\beta)\Sigma_2^{-1}m) - (1+\beta) \Sigma_3 \Sigma_2^{-1}(x-m)
    \\
    &=
    x - \Sigma_3\left(\Sigma_1^{-1} + (1+\beta)\Sigma_2^{-1}\right)x
    =0,
\end{align*}
using that $\Sigma_3 = \left(\Sigma_1^{-1} + (1+\beta)\Sigma_2^{-1}\right)^{-1}$.
This proves that $\rho^\star$ is a stationary solution of the Fokker--Planck equation \labelcref{eq:polarized_fuck_kerplunk_CBS}.
\end{proof}

\subsection{\rev Towards mean-field analysis\nc}
\label{sec:analysis}

In this section we will make some remarks on the analysis of the Fokker--Planck equation \labelcref{eq:polarized_fuck_kerplunk_CBO} for polarized CBO which we repeat here for convenience:
\begin{align*}
    \partial_t \rho_t(x) = \div\Big(\rho_t(x)(x-\wkmean[\rho_t](x))\Big) + \frac{\sigma^2}{2}\Delta\left(\rho_t(x)\abs{x - \wkmean[\rho](x)}^2\right).
\end{align*}
\rev 
We shall first explain why current analytical frameworks for the mean-field analysis of consensus-based optimization do not carry over to this equation.
Then, we shall prove convergence as $t\to\infty$ of solutions to this equation to a singular measure located at a minimizer, restricting ourselves to the the zero temperature limit $(\beta=0)$ and to sufficiently nice objective functions $V$.
\nc

Weak solutions of this equation are continuous curves of probability measures  $t\mapsto\rho_t$ such that
\begin{align*}
    \frac{\d}{\d t}
    \int 
    \phi(x) \d\rho_t(x)
    =
    &-
    \int
    \grad\phi(x)\cdot(x-\wkmean[\rho_t](x))
    \d\rho_t(x)\\
    &+
    \frac{\sigma^2}{2}
    \int 
    \Delta\phi(x) \abs{x-\wkmean[\rho_t](x)}^2\d\rho_t(x)
\end{align*}
holds true for all smooth and compactly supported test functions $\phi \in C^\infty_c(\nR^d)$.

Using the Leray--Schauder fixed point theorem, existence proofs for this equation without a kernel, i.e., $\kernel(x,y)=1$, were given in \cite{carrillo2018analytical,fornasier2021consensus_conv} under mild assumptions on the objective $V$ and for initial distributions with finite fourth-order moment.
Taking into account that 
\begin{align*}
    \kernel(x,y)\exp(-\beta V(y)) = 
    \exp\left(-\beta\left(V(y) - \frac{1}{\beta} \log\kernel(x,y)\right)\right)
\end{align*}
we expect that under reasonable Lipschitz-like assumptions on the logarithm of the kernel these arguments translate to our case, but we leave this for future work.
The biggest challenge is that for standard CBO $t\mapsto\wmean[\rho_t]$ is a continuous curve in $\nR^d$, whereas $t\mapsto\wkmean[\rho_t](\cdot)$ is not.
Rather, it is a curve in a space of vector fields.
This requires more sophisticated compactness arguments than the one made in \cite{fornasier2021consensus_conv,carrillo2018analytical}.

Besides existence, in the literature there exist two different approaches to proving formation of consensus around the global minimizer of $V$ for the Fokker--Planck equation \labelcref{eq:fuck_kerplunk_CBO} associated to standard CBO.
The first one was presented in \cite{carrillo2018analytical} and constitutes a two-step approach. 
First, they prove that the non-weighted standard variance
\begin{align*}
    V(\rho_t) := \int \abs{x - E(\rho_t)}^2\d \rho_t(x),\qquad\text{where }E(\rho_t) := \int y\d\rho_t(y),
\end{align*}
decreases to zero along solutions $\rho_t$ of \labelcref{eq:fuck_kerplunk_CBO}.
This obviously implies that $\rho_t$ converges to a Dirac measure $\delta_{\tilde x}$ concentrated on some point $\tilde x \in \nR^d$.
Second, the Laplace principle is invoked in order to conclude that $\tilde x$ lies close to the global minimizer of $V$ if $\beta$ is chosen sufficiently large.
This analytical approach uses heavily that the weighted mean $\wmean[\rho_t]$ for CBO \emph{does not} depend on the spatial variable~$x$ which is a linearity property that our weighted mean $\wkmean[\rho_t](x)$ does not enjoy. 
Furthermore, by design our method does in general not converge to a single Dirac mass and hence it's classical variance does not converge to zero.

A different approach is presented in \cite{fornasier2021consensus_conv} where the authors propose a more unified strategy. 
For this they fix a point $\hat{x}\in\nR^d$ (which later will be the global minimizer of $V$) and define the variance-type function
\begin{align}
    \nvar[\rho_t] := 
    \frac{1}{2}\int\abs{x-\hat{x}}^2\d\rho_t(x).
\end{align}
The authors note that $\nvar[\rho_t] = \frac12\mathcal{W}_2^2(\rho_t,\delta_{\hat{x}})$ and so convergence of the variance implies convergence of $\rho_t$ to $\delta_{\hat{x}}$ in the Wasserstein-2 distance.
They derive a differential inequality for $\nvar[\rho_t]$ which in combination with the Laplace principle allows them to show some kind of semi-convergence behavior, i.e., for every $\eps>0$ there exists $\beta>0$ such that $\nvar[\rho_t]$ decreases exponentially fast until it hits the threshold $\nvar[\rho_t]\leq\eps$.

When it comes to generalizing this approach to our setting there are two main obstacles.
First, in the case of several global minimizers $\{\hat{x}_i\st 1\leq i \leq N\}$ the Wasserstein-2 distance between $\rho_t$ and the empirical measure $\frac{1}{N}\sum_{i=1}^N\delta_{\hat{x}_i}$ does not equal $\frac1N\sum_{i=1}^N\int\abs{x-\hat{x}_i}^2\d\rho_t(x)$.
Indeed, the latter quantity is a very bad upper bound for the desired Wasserstein-2 distance since it is bounded from below by a positive number.

\rev In the following, we therefore present an alternative approach for analyzing convergence of the Fokker--Planck equation \labelcref{eq:polarized_fuck_kerplunk_CBO} which is based on choosing the Lyapunov function
\begin{align*}
    \L_\beta[\rho_t] := \frac{1}{2}\int\abs{x-\wkmean[\rho_t](x)}^2\d\rho_t(x)
\end{align*}
and the reasons for this choice will become clear in a moment.
We consider the setting of a Gaussian kernel $\kernel(x,y) := \exp\left(-\beta\frac{\abs{x-y}^2}{2\kappa^2}\right)$ of variance $\kappa^2/\beta$ for $\kappa>0$ and consider the limit as $\beta\to\infty$. 
The case of $\beta<\infty$ is then treated using the quantitative Laplace principle as, for instance, in \cite{fornasier2021consensus_conv}.
Let us assume that the support of $\rho$ equals $\nR^d$.
In this setting the Laplace principle implies that for $\beta\to\infty$ one has
\begin{align*}
    \wkmean[\rho](x) = \frac{\int y\exp\left(-\beta\left(\frac{\abs{x-y}^2}{2\kappa^2}+V(y)\right)\right)\d\rho(y)}{\int\exp\left(-\beta\left(\frac{\abs{x-y}^2}{2\kappa^2}+V(y)\right)\right)\d\rho(y)}
    \to
    \argmin_{y \in \nR^d}\frac{\abs{x-y}^2}{2\kappa^2}+V(y),
\end{align*}
where we assume that $\kappa$ is sufficiently small such that $y\mapsto\frac{\abs{x-y}^2}{2\kappa^2}+V(y)$ has a unique minimizer for every $x\in\nR^d$. 
Note that this is possible, for instance, if $V$ is $C^2$ and the smallest eigenvalue of its Hessian matrix is bounded from below.
Therefore, we consider the limiting dynamics as $\beta\to\infty$, governed by the drift field $x - p(x)$, where
\begin{align}\label{eq:proximal_operator}
    p(x) := \argmin_{y \in \nR^d}\frac{\abs{x-y}^2}{2\kappa^2}+V(y)
\end{align}
is known as the \emph{proximal operator} of $\kappa^2 V$.
Correspondingly, the Lyapunov function becomes
\begin{align*}
    \L[\rho_t] := L_\infty[\rho_t] := \frac{1}{2}\int\abs{x-p(x)}^2\d\rho_t(x)
\end{align*}
and we would like to emphasize that, using properties of the proximal operator, one has
\begin{align*}
    \L[\rho] = 0 
    \iff 
    x\in \argmin V\quad\text{for $\rho$-almost every }x\in\nR^d.    
\end{align*}
The following analysis treats the model case where the loss function $V$ is strongly convex, sufficiently smooth, and additionally satisfies a certain derivative bound in case we consider a Fokker--Planck equation with non-vanishing diffusion term.
While this is just the first step toward a comprehensive analysis of polarized CBO method for much larger classes functions, the following results introduce a set of important techniques which---we believe---will be a cornerstone for future analysis.
The core ideas of the following analysis are based on discussions of the first author with Massimo Fornasier and Oliver Tse and an extension to non-convex loss functions and $\beta<\infty$ is ongoing work.

We start with the following decay property of $\L[\rho_t]$ for solutions of the associated Fokker--Planck equation:
\begin{proposition}[Exponential decay of the Lypapunov function]\label[proposition]{prop:exp_cvgc}
    Let $t\mapsto\rho_t$ be a weak solution of the Fokker--Planck equation
    \begin{align}
        \partial_t \rho_t(x) = \div\Big(\rho_t(x)(x-p(x))\Big) + \frac{\sigma^2}{2}\Delta\left(\rho_t(x)\abs{x - p(x)}^2\right).
    \end{align}
    We pose the following assumptions on the loss function $V$:
    \begin{itemize}
        \item If $\sigma=0$ assume that $V\in C^2(\nR^d)$ and there exists $\mu>0$ such that
        \begin{align*}
            D^2 V \succcurlyeq \mu\mathbb I    
        \end{align*}
        \item If $\sigma>0$ assume additionally that $V\in C^3(\nR^d)$ and satisfies
        \begin{align*}
        \sup_{x\in\nR^d}\abs{Dv(x)}\abs{D^3V(x)}<\infty.
        \end{align*} 
    \end{itemize}
    Then there exists constants $C_1,C_2>0$ such that, if $\sigma<C_1$, it holds for all $t>0$ that
    \begin{align}
        \frac{\d}{\d t}\L[\rho_t] \leq -C_2\L[\rho_t] 
    \end{align}
    and consequently
    \begin{align}
        \L[\rho_t] \leq \L[\rho_0]\exp(-C_2 t).
    \end{align}
\end{proposition}
\begin{proof}
    Differentiating $\L[\rho_t]$ in time and using the weak form of the PDE implies
    \begin{align}\label{eq:derivative_L}
    \begin{split}
        \frac{\d}{\d t}\L[\rho_t]
        =
        -\int\nabla\frac{1}{2}\abs{x-p(x)}^2\cdot (x-p(x))\d\rho_t(x)
        +
        \frac{\sigma^2}{4}
        \int \Delta\abs{x-p(x)}^2\abs{x-p(x)}^2\d\rho_t(x).
    \end{split}
    \end{align}
    We continue by computing the derivatives that appear in this expression. 
    First, it holds
    \begin{align}\label{eq:del_i_square}
        \partial_i \frac{1}{2}\abs{x-p(x)}^2
        =
        \partial_i \sum_j\frac{1}{2}\abs{x_j-p_j(x)}^2
        =
        \sum_j(\delta_{ij}-\partial_i p_j(x))(x_j-p_j(x)).
    \end{align}
    Next we observe that by definition of the proximal operator $p(x)$ it holds
   \begin{align}\label{eq:OC}
        0 = p(x) - x + \kappa^2\nabla V(p(x)).
   \end{align}
   Differentiating this equation yields 
   \begin{align}\label{eq:diff_OC}
        0 = Dp(x) - \mathbb I + \kappa^2 D^2V(p(x)) Dp(x)
   \end{align}
   and therefore
   \begin{align}\label{eq:jacobian_prox}
        Dp(x) &= \left(\mathbb I + \kappa^2 D^2V(p(x))\right)^{-1},
   \end{align}
   which is a symmetric matrix.
   Using \labelcref{eq:del_i_square} we get
   \begin{align*}
       -\nabla\frac{1}{2}\abs{x-p(x)}^2\cdot(x-p(x))
       =
       -(\mathbb I - Dp(x))(x-p(x)) \cdot (x-p(x))
   \end{align*}
   For estimating this expression from above it suffices to bound the eigenvalues of $M:=\mathbb I - Dp(x)$ from below. 
   By assumption we have $D^2 V \succcurlyeq \mu\mathbb I$, which implies that $M\succcurlyeq\left(1-\frac{1}{1+\kappa^2\mu}\right)\mathbb I = \frac{\kappa^2\mu}{1+\kappa^2\mu}\mathbb I$ and therefore we can bound the first term in \labelcref{eq:derivative_L}
   \begin{align}\label{eq:bound_drift}
       -\nabla\frac{1}{2}\abs{x-p(x)}^2\cdot(x-p(x))
       \leq
       -\frac{\kappa^2\mu}{1+\kappa^2\mu}
       \abs{x-p(x)}^2.
   \end{align}
   If $\sigma=0$ we can already conclude the proof, using Gronwall's inequality.
   For $\sigma>0$ we have to bound the second term in \labelcref{eq:derivative_L}, coming from the diffusion.
   Using \labelcref{eq:del_i_square} and the product rule it follows
   \begin{align}\label{eq:del_del_i_square}
   \begin{split}
       \partial_i^2\frac{1}{2}\abs{x-p(x)}^2
        =
        -\sum_j\partial_i^2 p_j(x)(x_j-p_j(x)) 
        +
        \sum_j(\delta_{ij}-\partial_i p_j(x))(\delta_{ij}-\partial_i p_j(x)).
   \end{split}
   \end{align}
   Consequently, we obtain
   \begin{align*}
       \Delta\frac12\abs{x-p(x)}^2
       &=
       \sum_i \partial_i^2\frac{1}{2}\abs{x-p(x)}^2
       \\
       &=
       -\sum_{i,j}\partial_i^2 p_j(x)(x_j-p_j(x)) + \sum_{i,j}(\delta_{ij}-\partial_i p_j(x))(\delta_{ij}-\partial_i p_j(x))
       \\
       &=
       -\sum_{i,j}\partial_i^2 p_j(x)(x_j-p_j(x))
       +
       d
       -2\sum_{i}\partial_i p_i(x) 
       +
       \sum_{i,j}
       \partial_i p_j(x)
       \partial_i p_j(x)
       \\
       &=
       -\sum_{i,j}\partial_i^2 p_j(x)(x_j-p_j(x))
       +
       d
       -2\tr(Dp(x))
       +
       \tr(Dp(x)Dp(x)^T)
       \\
       &=
       -\sum_{i,j}\partial_i^2 p_j(x)(x_j-p_j(x))
       +
       d
       +
       \tr(Dp(x)(Dp(x)^T-2\mathbb I)).
   \end{align*}
   It also holds $0\preccurlyeq Dp(x) \preccurlyeq \mathbb I$ and therefore $2\mathbb I - Dp(x)\succcurlyeq\mathbb I$.
   This allows us to estimate
   \begin{align*}
       \tr(Dp(x)(Dp(x)^T-2\mathbb I))
       =
       -\tr(Dp(x)(2\mathbb I-Dp(x)))
       \leq 
       -\tr(Dp(x)) \leq 0.
   \end{align*}
   Going back to the previous formula for the Laplacian, we obtain the estimate
   \begin{align}\label{eq:delta_2}
       \Delta\frac12\abs{x-p(x)}^2
       \leq
       -\sum_{i,j}\partial_i^2 p_j(x)(x_j-p_j(x))
       +
       d
   \end{align}
   and it remains to bound the first term. 
   For this we need second derivatives of $p(x)$. 
   Writing \labelcref{eq:diff_OC} in coordinates gives
   \begin{align*}
       \partial_i p_j(x) = \delta_{ij} - \kappa^2 \sum_{r}\partial_{r}\partial_j V(p(x))\partial_i p_r(x).
   \end{align*}
   Taking another derivative with respect to the $i$-th variable and using the preduct rule yields
   \begin{align*}
       \partial_i^2 p_j(x) = -\kappa^2\sum_r\partial_r\partial_j V(p(x))\partial_i^2 p_r(x)
       -
       \kappa^2
       \sum_{r,s}\partial_s\partial_r\partial_jV(p(x))\partial_i p_s(x)\partial_j p_r(x).
   \end{align*}
   We have to solve this equation for the second derivatives of $p$ for which we define the matrix $A=A_{ij} := \partial_i^2 p_j(x)$ and the matrix $B = B_{ij} := \sum_{r,s}\partial_s\partial_r\partial_jV(p(x))\partial_i p_s(x)\partial_j p_r(x)$.
   Then the previous equation is equivalent to the linear system
   \begin{align*}
       A = -\kappa^2 A D^2V(p(x)) - \kappa^2 B
   \end{align*}
   which is solved by
   \begin{align}\label{eq:A}
       A = -\kappa^2 B (\mathbb I + \kappa^2 D^2V(p(x)))^{-1}.
   \end{align}
   Using the definition of $A$ and \labelcref{eq:delta_2} we get
   \begin{align*}
       \Delta\frac12\abs{x-p(x)}^2
       \leq 
       -\kappa^2
       \sum_{i,j}
       A_{ij}\partial_j V(p(x)) + d
   \end{align*}
   and it remains to uniformly bound the first term.
      
   Using \labelcref{eq:A}, the definition of $B$ as well as \labelcref{eq:jacobian_prox}, we get the following estimate in terms of the matrix/tensor norms
   \begin{align*}
       \Delta\frac12\abs{x-p(x)}^2
       \leq 
       C_d
       \kappa^4
       \abs{D^3V(p(x))}\abs{\nabla V(p(x))}
       \abs{\left(\mathbb I+\kappa^2 D^2V(p(x))\right)^{-1}}^3
       + d,
   \end{align*}
   where $C_d$ is a dimensional constant.
   By assumption the right hand side is uniformly bounded by some constant $C>0$ and going back to \labelcref{eq:derivative_L} we obtain
   \begin{align*}
       \frac{\d}{\d t}\L[\rho_t] 
       &\leq 
       -\frac{\kappa^2\mu}{1+\kappa^2\mu}\int\abs{x-p(x)}^2\d\rho_t(x)
       +
       \frac{\sigma^2}{2}
       C
       \int \abs{x-p(x)}^2\d\rho_t(x)
       \\
       &=
       -\left(\frac{2\kappa^2\mu}{1+\kappa^2\mu}-\sigma^2 C\right)\L[\rho_t].
   \end{align*}
   Since $\mu>0$, we can choose $\sigma>0$ sufficiently small---for instance, $\sigma^2<\frac{1}{C}\frac{\kappa^2\mu}{1+\kappa^2\mu}$---such that the brackets are strictly positive. 
   Then we can conclude the proof using Gronwall's inequality.
\end{proof}
\begin{example}[The one-dimensional case]
   In the case of one spatial dimension $d=1$, we can bound the Laplacian $\Delta\frac12\abs{x-p(x)}^2$---which in this case is just the second derivative---more accurately. 
   In this case $B=V'''(p(x))(p'(x))^2$ and hence $A=-\kappa^2\frac{V'''(p(x))(p'(x))^2}{1+\kappa^2V''(p(x))}$.
   Plugging in $p'(x)=\frac{1}{1+\kappa^2V''(p(x))}$ we obtain
   \begin{align*}
       A = -\kappa^2\frac{V'''(p(x))}{(1+\kappa^2V''(p(x)))^3}.
   \end{align*}
   Hence, we obtain the estimate
   \begin{align*}
       \Delta\frac12\abs{x-p(x)}^2
       \leq 
       \kappa^4
       \frac{V'''(p(x))V'(p(x))}{(1+\kappa^2V''(p(x)))^3}
       + 1.
   \end{align*}
   We give three examples of loss functions $V$ which satisfy all assumptions of \cref{prop:exp_cvgc}.
   The first one is the quadratic loss $V_1(x)=ax^2+bx+c$ with $a>0,b,c\in\nR$ for which it holds $V_1'''=0$ and hence the Laplacian is bounded by $1$.
   The second one is $V_2(x)=x^2+\ln(x+\sqrt{1+x^2})$ for which we have (the loose bound) $\abs{V'''V'}\leq 1$ and so a valid bound for the Laplacian is $\kappa^4+1$.
   The third loss is $V_3(x)=x^2+x+1-(x+1)\ln(x+1)$ for $x\geq 0$, extended to an even function on $\nR$. 
   Also here one has (the loose) bound $\abs{V'''V'}\leq 1$ and again a valid bound for the Laplacian is $\kappa^4+1$.
\end{example}
\rev
Next we prove a compactness property for measures for which $\L[\rho_t]$ is uniformly bounded in time.
For this we also require that the loss $V$ has a Lipschitz continuous gradient, which is a standard assumption in nonlinear optimization.
\begin{proposition}\label[proposition]{prop:compactness}
    Let $V \in C^2(\nR^d)$ satisfy $\mu\mathbb I\preccurlyeq D^2 V \preccurlyeq L\mathbb I$ for some $0<\mu\leq L<\infty$ and assume that
    \begin{align*}
        \sup_{t>0}\L[\rho_t] < \infty.
    \end{align*}
    Then there exists a subsequence $(\rho_{t_n})_{n\in\N}$ for $t_n\to\infty$ as $n\to\infty$ which converges to a probability measure $\rho_\infty$ in the Wasserstein-2 distance.
\end{proposition}
\begin{proof}
    Let $x^\ast:=\argmin V$.
    By assumption it holds
    \begin{align*}
        \frac{\mu}{2}\abs{p(x)-x^\ast}^2
        +
        \langle 
        \nabla V(p(x)), p(x)- x^\ast\rangle 
        +
        V(p(x))
        \leq 
        V(x^\ast) \leq V(p(x))
    \end{align*}
    which implies
    \begin{align*}
        \frac{\mu\kappa^2}{2}\abs{p(x)-x^\ast}^2
        \leq 
        \kappa^2\langle 
        \nabla V(p(x)), x^\ast - p(x)\rangle.
    \end{align*}
    By definition of $p(x)$ it holds $\kappa^2\nabla V(p(x))=x-p(x)$ and therefore 
    \begin{align*}
        \frac{\mu\kappa^2}{2}\abs{p(x)-x^\ast}^2
        \leq
        \langle x-p(x),x^\ast-p(x)\rangle
        \leq 
        \abs{x-p(x)}\abs{x^\ast-p(x)}.
    \end{align*}
    Hence, we have showed that
    \begin{align*}
        \abs{p(x)-x^\ast} \leq \frac{2}{\mu\kappa^2} \abs{x-p(x)}.
    \end{align*}
    Using this as well as the inequality $\abs{a+b}^2\leq 2\abs{a}^2+2\abs{b}^2$ we obtain
    \begin{align}\label{eq:moment_bound}
    \begin{split}
        \int\abs{p(x)}^2\d\rho_t(x)
        &\leq 
        2\int\abs{x^\ast-p(x)}^2\d\rho_t(x)
        +
        2\abs{x^\ast}^2
        \\
        &\leq 
        \frac{8}{\mu^2\kappa^4}
        \int\abs{x-p(x)}^2\d\rho_t(x)
        +
        2\abs{x^\ast}^2.
    \end{split}        
    \end{align}
    By assumption one has
    \begin{align*}
        \sup_{t>0}\int\abs{x-p(x)}^2\d\rho_t(x) < \infty
    \end{align*}
    which, together with change of variables and \labelcref{eq:moment_bound}, implies
    \begin{align*}
        \sup_{t>0}\int\abs{y}^2\d(p_\sharp\rho_t)(y)=
        \sup_{t>0}\int\abs{p(x)}^2\d\rho_t(x)<\infty.
    \end{align*}
    Therefore, by compactness a subsequence $(p_\sharp\rho_{t_n})_{n\in\N}$ of the push-forward measures converges to some probability measure $\sigma_\infty$ in the Wasserstein-2 distance as $n\to\infty$.
    
    First, we claim that under our assumptions the map $x\mapsto p(x)$ is a biLipschitz homeomorphism with Lipschitz continuous inverse $y\mapsto p^{-1}(y):=y+\kappa^2\nabla V(y)$.
    The fact that these two maps are Lipschitz continuous is obvious, noting that the proximal operator of a convex function is $1$-Lipschitz and the map $p^{-1}$ is Lipschitz because $\nabla V$ is Lipschitz.
    Furthermore, we observe that
    \begin{align*}
        p^{-1}(p(x)) = p(x) + \kappa^2\nabla V(p(x)) = x
    \end{align*}
    by definition of the proximal operator so $p^{-1}$ is a left inverse. 
    To see that also $p(p^{-1}(x))=x$ we note that
    \begin{align*}
        p(p^{-1}(x)) = \argmin_{y\in\nR^d}\frac{\abs{y-x-\kappa^2\nabla V(x)}^2}{2\kappa^2}+V(y).
    \end{align*}
    This is a strongly convex optimization problem and hence the optimality conditions
    \begin{align*}
        0 = y-x-\kappa^2\nabla V(x) + \kappa^2\nabla V(y)
    \end{align*}
    are necessary and sufficient. 
    They have the (unique) solution $y=x$ and therefore $p(p^{-1}(x))=x$ which shows that $p^{-1}$ is truely the inverse of $p$.
    
    Next, we claim that the measures $\rho_{t_n}$ converge to $\rho_\infty := p^{-1}_\sharp\sigma_\infty$ in the Wasserstein-2 distance as $n\to\infty$.
    To see this, let $\pi_n$ denote an optimal coupling of $p_\sharp\rho_{t_n}$ and $\sigma_\infty$ such that
    \begin{align*}
        W_2(p_\sharp\rho_{t_n},\sigma_\infty)
        =
        \iint\abs{x-y}^2\d\pi_n(x,y).
    \end{align*}
    We define $\tilde\pi_n := (p^{-1}\times p^{-1})_\sharp\pi_n$ and first argue that this a coupling of $\rho_{t_n}$ and $\rho_\infty$.
    To see this, we compute
    \begin{align*}
        \tilde\pi_n(A\times\nR^d)
        =\pi_n(p(A)\times p(\nR^d))
        =\pi_n(p(A)\times \nR^d)
        =
        p_\sharp\rho_{t_n}(p(A))
        =
        \rho_{t_n}(A)
    \end{align*}
    for every Borel set $A\subset\nR^d$, where we used that $p$ is invertible and that the first marginal of $\pi_n$ is $p_\sharp\rho_{t_n}$. 
    Hence, the first marginal of $\tilde\pi_n$ is $\rho_{t_n}$.
    Similarly, the second marginal can be shown to be $\rho_\infty$ and hence $\pi_n$ is a coupling.    
    We therefore obtain
    \begin{align*}
        W_2(\rho_{t_n},\rho_\infty)
        &\leq 
        \iint\abs{x-y}^2\d\tilde\pi_n(x,y)
        \\
        &=
        \iint\abs{p^{-1}(x)-p^{-1}(y)}^2\d\pi_n(x,y)
        \\
        &\leq 
        (1+\kappa^2 L)^2
        \iint\abs{x-y}^2\d\pi_n(x,y)
        \\
        &=
        W_2(p_\sharp\rho_{t_n},\sigma_\infty)\to 0
        ,\qquad
        \text{as }n\to\infty,
    \end{align*}
    where we used the Lipschitzness of $\nabla V$ which implies Lipschitzness of $p^{-1}$.
\end{proof}
\begin{theorem}\label[theorem]{thm:convergence_W_2}
Under the conditions of \cref{prop:exp_cvgc,prop:compactness} and assuming $\L[\rho_0]=0$ it holds that
\begin{align*}
    W_2(\rho_t,\delta_{x^\ast}) \to 0\quad\text{as }t\to\infty,
\end{align*}
where $x^*:=\argmin V$.
\end{theorem}
\begin{proof}
    By \cref{prop:exp_cvgc} it holds $\L[\rho_t]\to 0$ and, in particular, $t\mapsto\L[\rho_t]$ is uniformly bounded.
    Hence, one can apply \cref{prop:compactness} to obtain that a subsequence of $\rho_t$ converges to some probability measure $\rho_\infty$ in the Wasserstein-2 distance and hence also weakly.
    Since $\rho\mapsto\L[\rho]$ is an integral of the continuous and lower-bounded function $x\mapsto\frac{1}{2}\abs{x-p(x)}^2$ against $\rho$, it is lower semicontinuous with respect to weak convergence of measures.
    It follows $\L[\rho_\infty]=0$ and hence $\rho_\infty=\delta_{x^\ast}$, where $x^\ast$ is the global minimizer of $V$.
    The uniqueness of the minimizer implies that the whole sequence $\rho_t$ converges to $\delta_{x^\ast}$ and this concludes the proof.
\end{proof}
\nc

\section{Numerical examples}\label{sec:numerics}

In this section we evaluate the numerical performance of the proposed algorithms. In all our experiments we chose a time step parameter of $\d t=0.01$. The code to reproduce all numerical experiments is available on \texttt{GitHub}\footnote{\url{https://github.com/TimRoith/PolarCBO}}.

\subsection{Unimodal Ackley function}

\begin{figure}[thb]
\begin{subfigure}{.32\textwidth}%
\includegraphics[width=\textwidth]{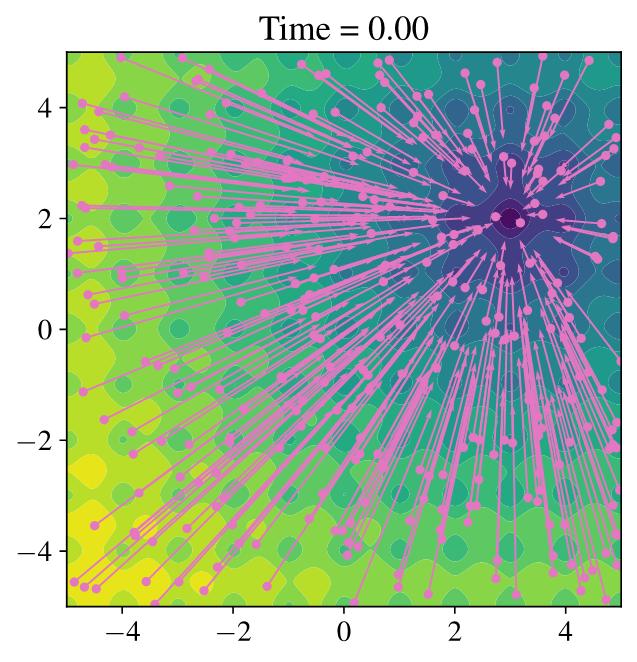}
\end{subfigure}%
\hfill%
\begin{subfigure}{.32\textwidth}%
\includegraphics[width=\textwidth]{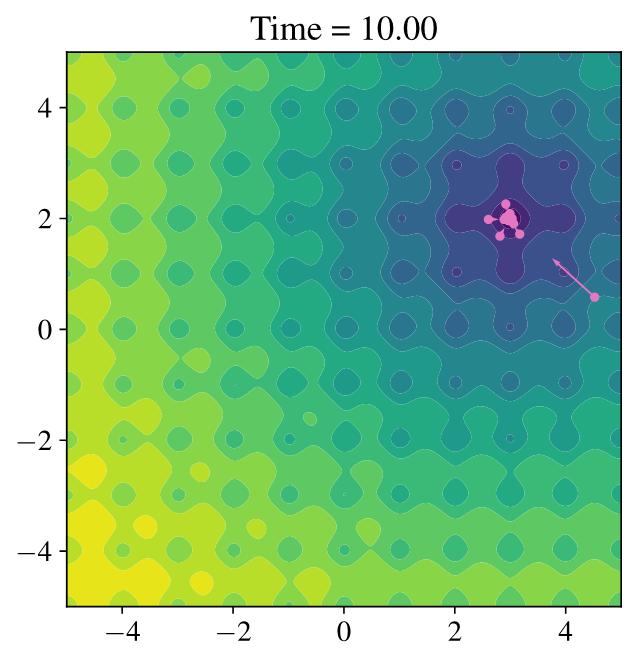}
\end{subfigure}%
\hfill%
\begin{subfigure}{.32\textwidth}%
\includegraphics[width=\textwidth]{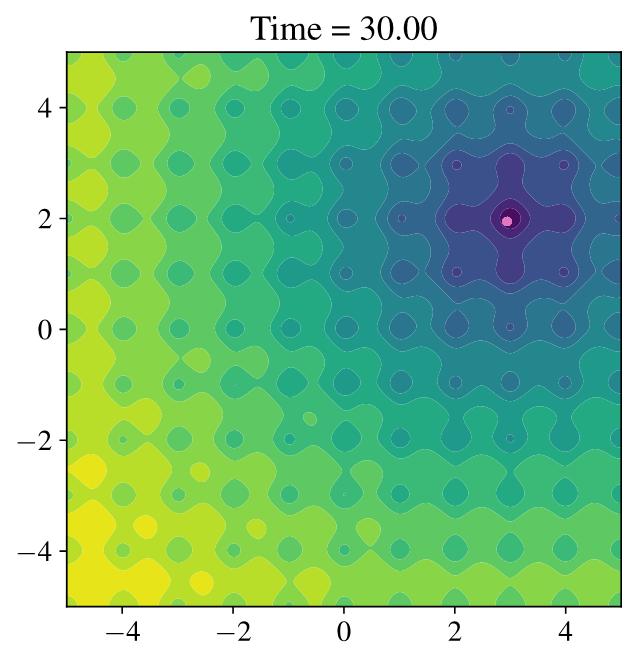}
\end{subfigure}%
\caption{Dynamics of standard CBO for minimizing the Ackley function. 
The points mark particle locations, the arrows the drift field towards the shared weighted mean.\label{fig:Ackley_CBO}}
%
%
%
\hfill%
\begin{subfigure}{.32\textwidth}
\includegraphics[width=\textwidth]{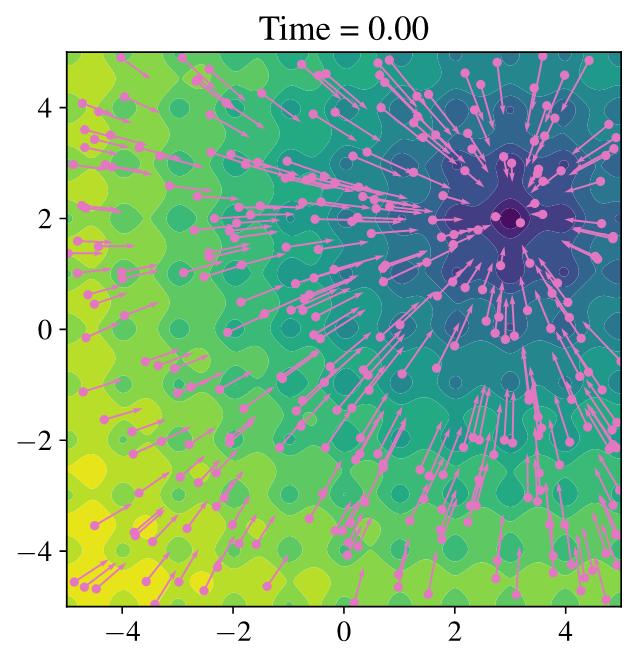}
\end{subfigure}%
\hfill%
\begin{subfigure}{.32\textwidth}%
\includegraphics[width=\textwidth]{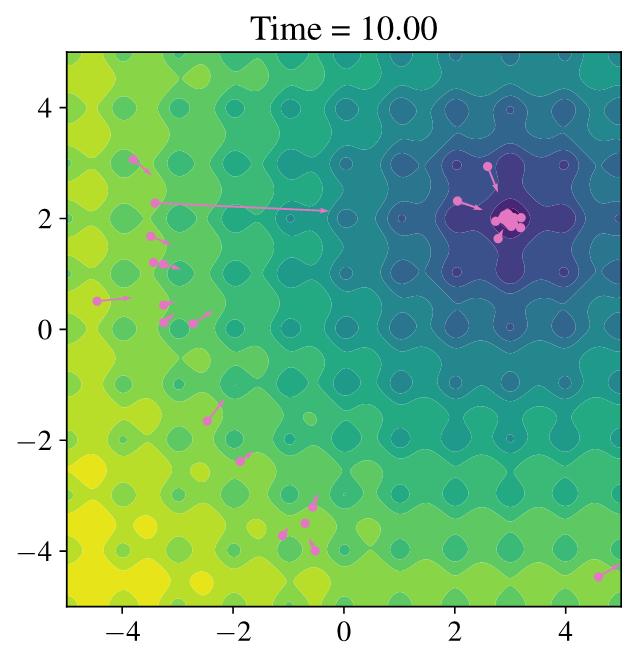}
\end{subfigure}%
\hfill%
\begin{subfigure}{.32\textwidth}%
\includegraphics[width=\textwidth]{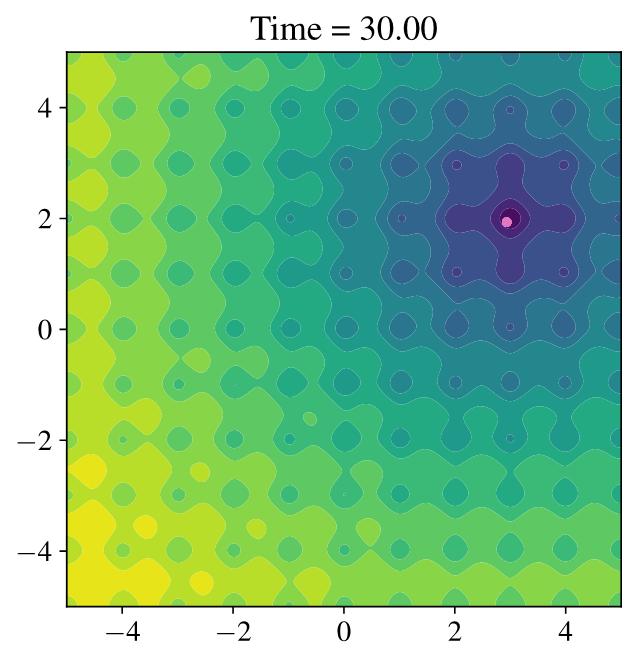}
\end{subfigure}%
\caption{Dynamics of the proposed polarized CBO for minimizing the Ackley function.
The points mark particle locations, the arrows the drift field towards the individual weighted means.\label{fig:Ackley_KCBO}}
\end{figure}
In the first example we perform a consistency check for our method for finding the unique global minimum of the Ackley function \cite{ackley2012}, defined as 
\begin{align}\label{eq:Ackley_fct}
    A(x) := -20 \exp\left(-\frac{0.2}{\sqrt{d}} \abs{x}\right) -
    \exp\left(\frac{1}{d}\sum_{n=1}^d\cos(2\pi x_n)\right) + e + 20.
\end{align}
This function has a global minimum at $0 \in\nR^d$ with $A(0)=0$ and in this experiment we choose $d=2$ and minimize the shifted version $V(x) := A(x-(3,2))$ which has its global minimum at $(3,2)\in\nR^2$. 
We compare the dynamics of standard CBO with our proposed polarized variant at three different time points in \cref{fig:Ackley_CBO,fig:Ackley_KCBO}.
We use the the Gaussian kernel $\kernel(x,y) := \exp\left(-\frac{\abs{x-y}^2}{2\kappa^2}\right)$ with standard deviation $\kappa=\infty$ for standard CBO and $\kappa=1$ for polarized CBO.
Furthermore, we choose $\beta=1$.
In this \rev simple\nc\ situation with a unique global minimum we observe that both standard and polarized CBO find the global minimum and do not get stuck in local minima.
Notably, the polarized variant converges slightly slower than standard CBO which is due to the localization effect of the kernel with a relatively small standard deviation.

\subsection{Multimodal Rastrigin function}

In this example we evaluate different choices of kernel functions for minimizing a Rastrigin-type function with three global minima. 
The original Rastrigin function \cite{rastrigin1974systems} on $\nR^d$ is defined as
\begin{align}
    R(x) := 10 d + \sum_{n=1}^d(x_n^2 - 10\cos(2\pi x_n))
\end{align}
and has a global minimum at $x=0$ with $U(0)=0$.
In this experiment we choose $d=2$ and minimize the product
\begin{align*}
    V(x) = \frac18 R(x) \, R(x-(3,2)) \, R(x+(1,3.5))
\end{align*}
which has three global minima.
Note that this function is very non-convex and at the same time extremely flat around its minima, see \cref{fig:Rastrigin} for a surface plot.

\begin{figure}[h]
    \centering
    \includegraphics[width=0.5\textwidth,trim=0.8cm 0cm 0cm 0.8cm,clip]{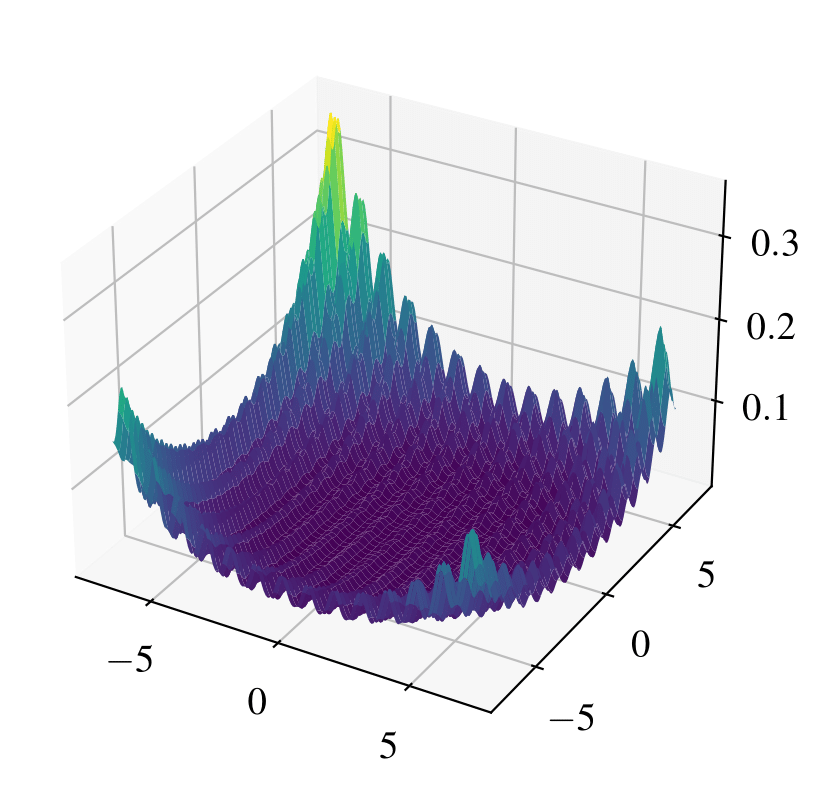}
    \caption{A variant of the Rastrigin function with three global minima.}
    \label{fig:Rastrigin}
\end{figure}

We consider the three different kernels
\begin{alignat}{2}
    \kernel(x,y) &= \exp\left(-\frac{\abs{x-y}^2}{2\kappa^2}\right)\quad&&\text{Gaussian kernel},
    \\
    \kernel(x,y) &= \exp\left(-\frac{\abs{x-y}}{\kappa}\right)\quad&&\text{Laplace kernel},
    \\
    \kernel(x,y) &= 1_{\abs{x-y}\leq\kappa}(x,y)\quad&&\text{bounded confidence kernel},
\end{alignat}
and the corresponding results of our method are depicted in \cref{fig:Rastrigin_kernels}.
Again, we choose $\beta=1$.
The kernel parameters $\kappa$ were chosen sufficiently small for the methods to detect all three minima, which are marked with blue diamonds.
While the Gaussian and the Laplace kernel work similarly well---note however that the Laplace kernel needs a much smaller value of $\kappa$ than then Gaussian---the bounded confidence kernel (cf. the discussion in \cref{ss:polarized_CBO}) works suboptimally for the task of minimization.
While it manages to detect all three minima a lot of particle get stuck in suboptimal consensus points which can be explained by the fact that the kernel has compact support.

\begin{figure}[thb!]
\begin{subfigure}{\textwidth}%
\includegraphics[width=0.3\textwidth]{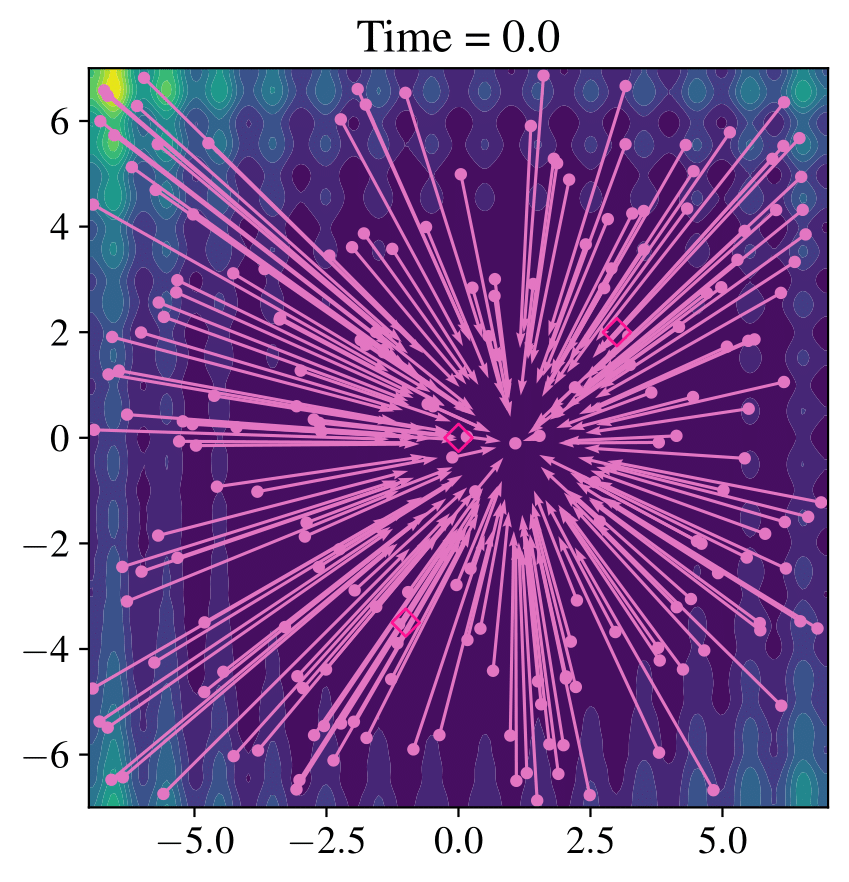}\hfill%
\includegraphics[width=0.3\textwidth]{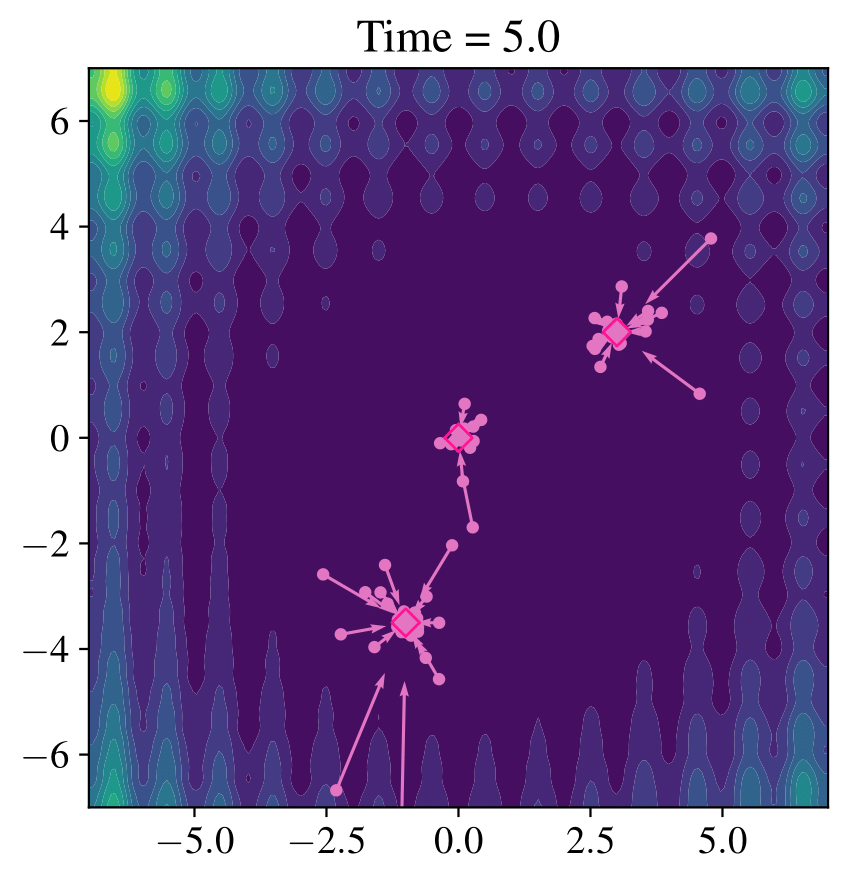}\hfill%
\includegraphics[width=0.3\textwidth]{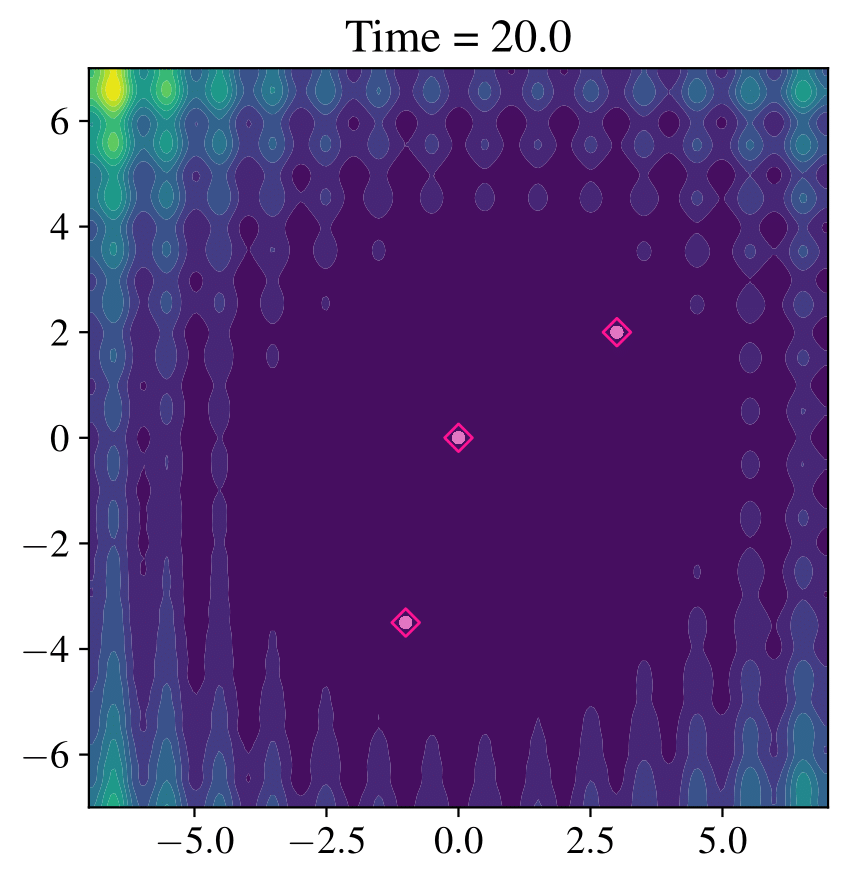}
\caption{Gaussian kernel with $\kappa=0.5$}
\end{subfigure}%
\\
\begin{subfigure}{\textwidth}%
\includegraphics[width=0.3\textwidth]{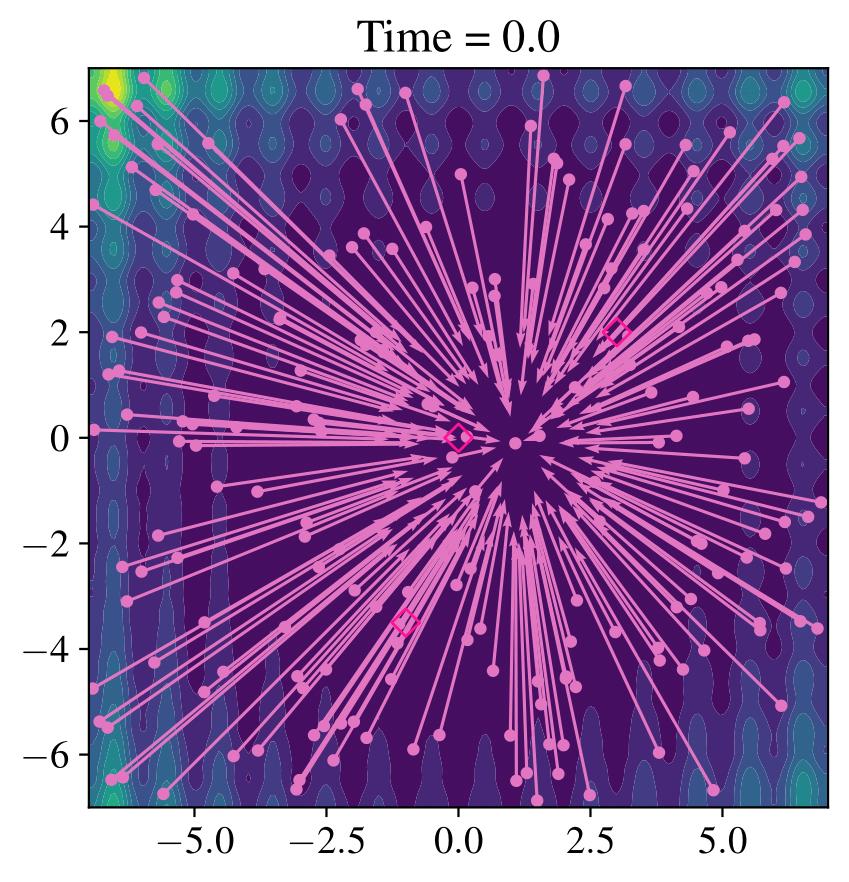}\hfill%
\includegraphics[width=0.3\textwidth]{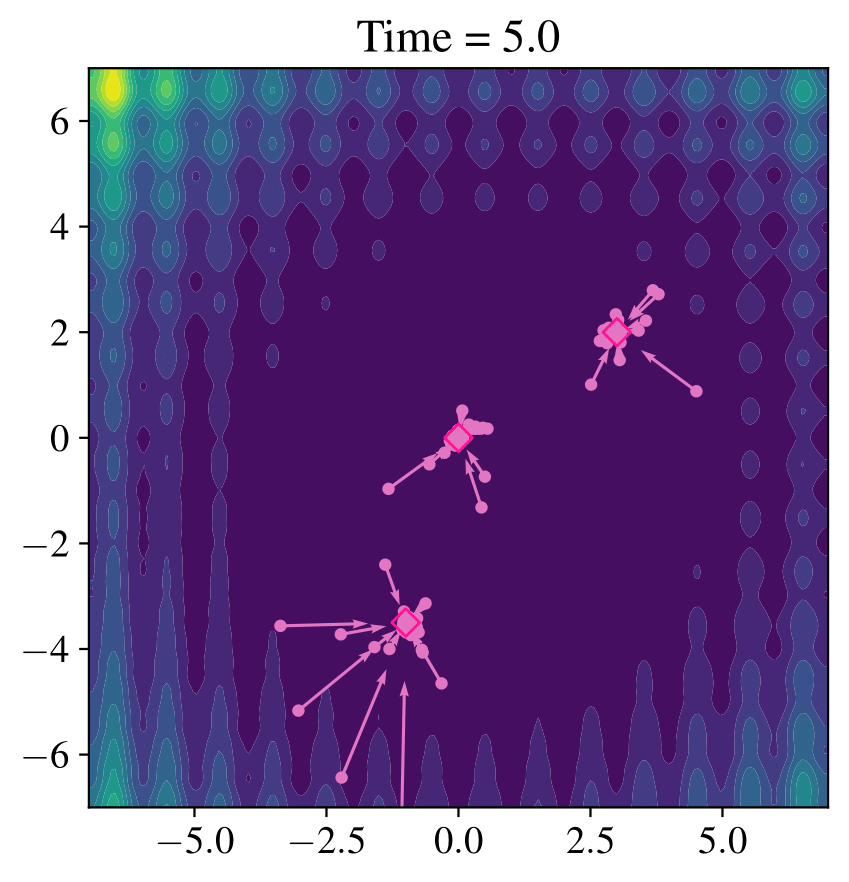}\hfill%
\includegraphics[width=0.3\textwidth]{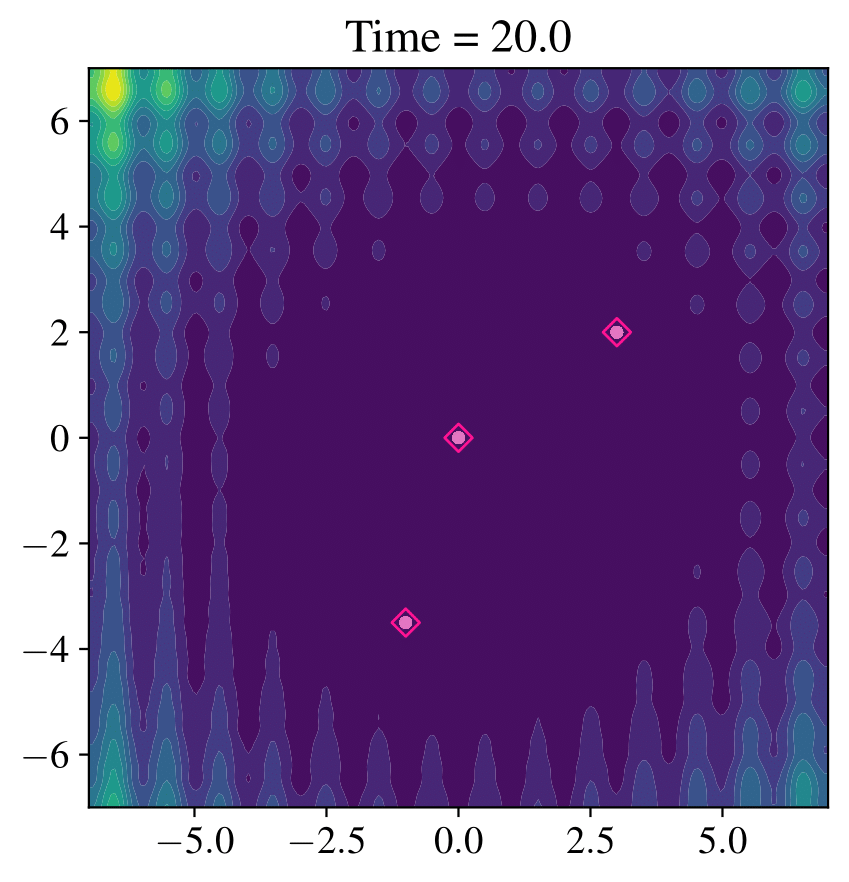}
\caption{Laplace kernel with $\kappa=0.05$.}
\end{subfigure}%
\\
\begin{subfigure}{\textwidth}%
\includegraphics[width=0.3\textwidth]{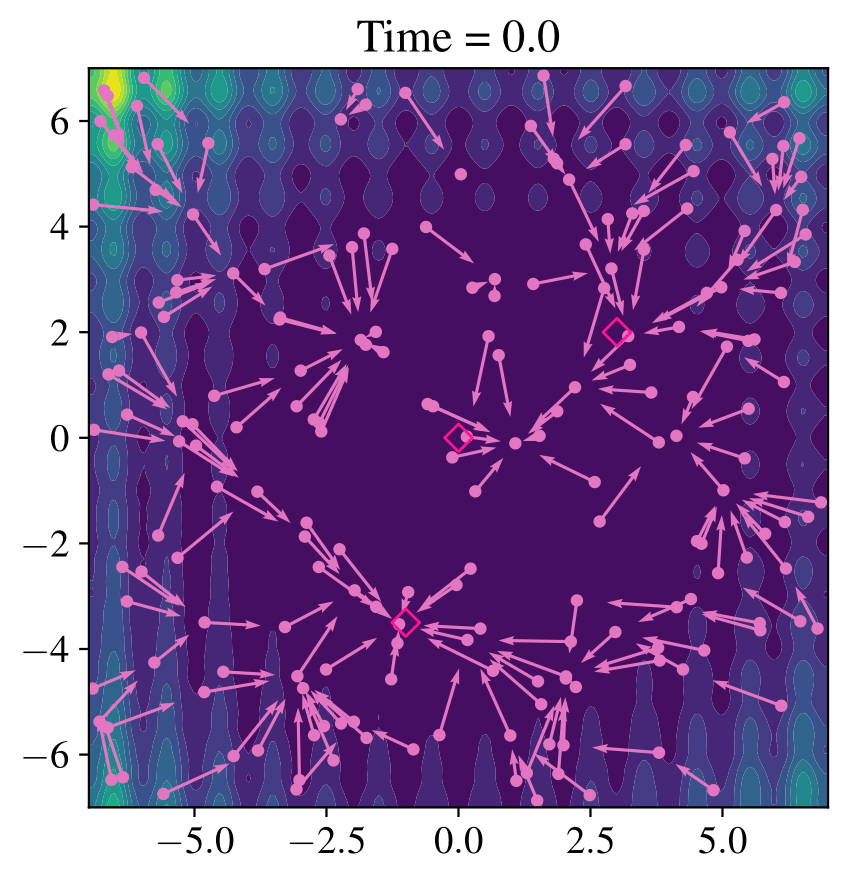}\hfill%
\includegraphics[width=0.3\textwidth]{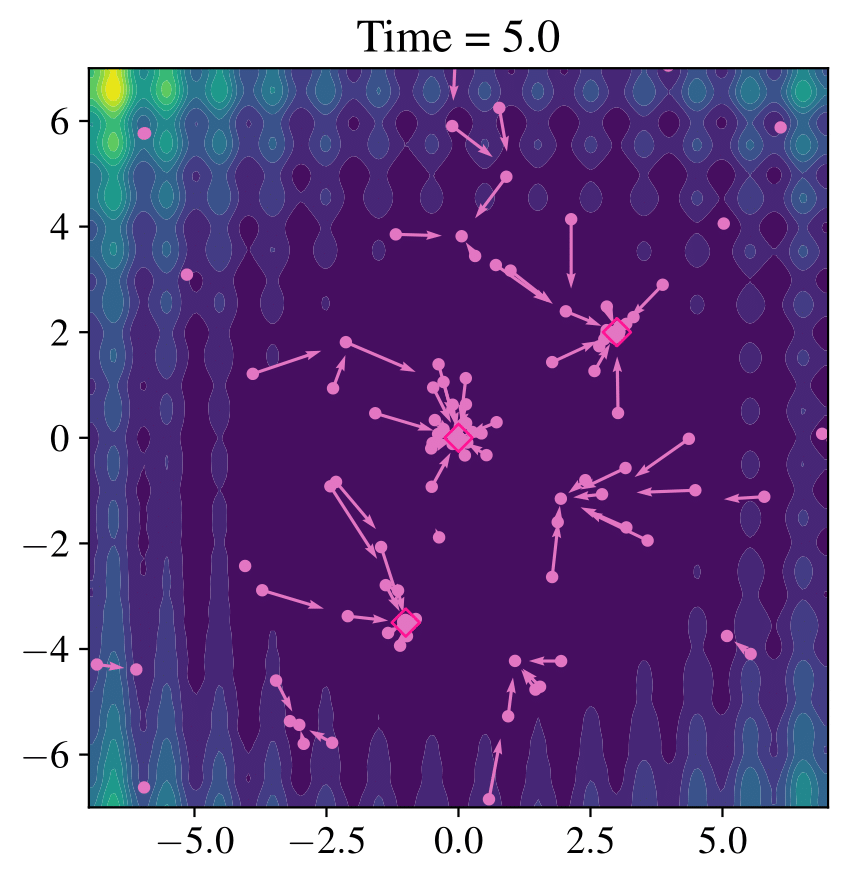}\hfill%
\includegraphics[width=0.3\textwidth]{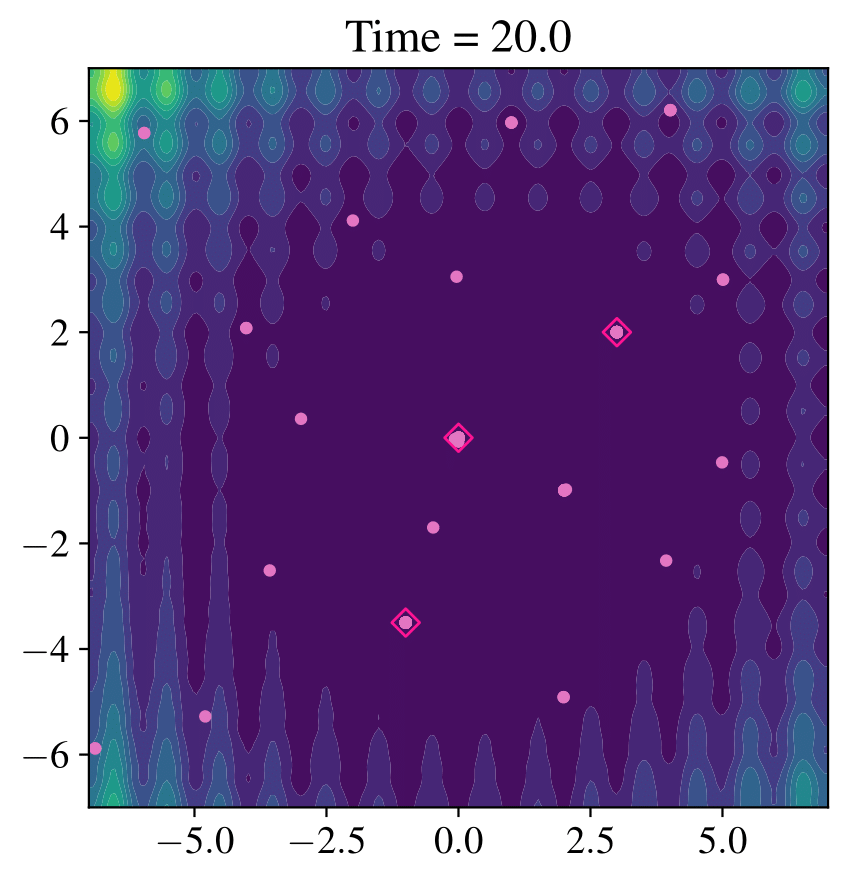}
\caption{Bounded confidence kernel with $\kappa=2$.}
\end{subfigure}%
\caption{Dynamics of polarized CBO with different kernels for minimizing a Rastrigin-type function with three global minima, marked by red diamonds. 
For all kernels all minima were detected.
Gaussian and Laplace kernel work especially well, whereas the bounded confidence kernel generates too many consensus points due to its compact support.}\label{fig:Rastrigin_kernels}
\end{figure}

\subsection{Multimodal Ackley function}

We proceed with quantitative evaluations of our method for a multimodal version of the Ackley function \labelcref{eq:Ackley_fct}, defined as
\begin{align*}
    V(x) := \prod_{i=1}^N A(x - z_i),
\end{align*}
where $\{z_i\in\nR^d\st i=1,\dots,N\}$ are points which constitute the global minimizers of $V$, and $U$ is the standard Ackley function defined in \labelcref{eq:Ackley_fct}.
In \cref{tab:multimodal_ackley} we plot how many of the three minima in dimension $d=2$ were detected by the proposed polarized CBO method with a Gaussian kernel and different values of the standard deviation $\kappa$.
That means, we show the percentage of runs that detected at least $1$, $2$, or $3$ minima. Here, we employed the standard noise model as specified in \labelcref{eq:our_CBO} with $\sigma=1.0$ and $\beta=1.0$.
For completeness we also include results for standard CBO (i.e., $\kappa=\infty$) which by definition can detect at most one minimum.

We say that a method detects a minimum if at convergence there exists a weighted mean $\wkmean[\rho](x)$ which is closer than $0.25$ in the infinity norm to the minimum.
For standard CBO where $\wkmean[\rho](x) = \wmean[\rho]$ for all $x$ this coincides with the definition of success in \cite{carrillo2022consensus_sampling,pinnau2017consensus}. For our experiments we employ $N=3$ different minima $z_1,z_2,z_3\in\nR^{d}$ with 
\begin{align*}
\resizebox{0.91\hsize}{!}{%
$
(z_1)_i :=
\begin{cases}
-2&\text{ if } i\operatorname{mod}2 = 0,\\
\phantom{-}1&\text{ else, }
\end{cases}
\quad
(z_2)_i :=
\begin{cases}
\phantom{-}2&\text{ if } i\operatorname{mod}2 = 0,\\
-1&\text{ else, }
\end{cases}
\quad
(z_3)_i :=
\begin{cases}
-1&\text{ if } i\operatorname{mod}2 = 0,\\
-3&\text{ else}
\end{cases}
$}
\end{align*}
for $i=1,\ldots, d$. While polarized CBO works very well in the two-dimensional example, \cref{tab:multimodal_ackley_d_10} shows that in dimension $d=10$, it fails to detect more than one minimum. However, the cluster method from \labelcref{alg:CCBO} exhibits significantly improved behavior and manages to detect all three minima frequently. Here, we employed the coordinate-wise noise model from \labelcref{eq:compNoise} with $\sigma=7.5$. Additionally we employed a simple scheduling for the parameter $\beta$, where we start with $\beta=30$ and increase it in each step via $$\beta\gets 1.01\cdot \beta$$ up to a limit of $\beta_\text{max}=10^7$. Here, one could potentially employ more sophisticated approaches as proposed in \cite{carrillo2022consensus_sampling}. For the cluster-based methods we chose $\alpha=5.0$ in \cref{alg:CCBO}.

\begin{table}[t!]
\centering%
\setlength{\extrarowheight}{3pt}%
\begin{adjustbox}{width=\columnwidth,center}
\parbox{2mm}{\rotatebox[origin=c]{90}{\hspace{-25pt}Polarized}}
\hspace{5pt}
\begin{tabular}{|l|c|c|c||c|c|c||c|c|c||c|c|c|}
\hline
 &
  \multicolumn{3}{c||}{$J=25$} &
  \multicolumn{3}{c||}{$J=50$} &
  \multicolumn{3}{c||}{$J=100$} &
  \multicolumn{3}{c|}{$J=200$}\\
\hline
$\#$ minima &
$\geq 1$ & $\geq 2$ & $\geq 3$ &
$\geq 1$ & $\geq 2$ & $\geq 3$ &
$\geq 1$ & $\geq 2$ & $\geq 3$ &
$\geq 1$ & $\geq 2$ & $\geq 3$ \\
\hline
$\kappa = 0.1$ & $33\%$ & $7\%$ & $0\%$ & $86\%$ & $59\%$ & $24\%$ & $100\%$ & $96\%$ & $67\%$ & $100\%$ & $100\%$ & $97\%$\\ 
$\kappa = 0.5$ & $100\%$ & $62\%$ & $5\%$ & $100\%$ & $78\%$ & $18\%$ & $100\%$ & $93\%$ & $41\%$ & $100\%$ & $100\%$ & $84\%$\\ 
$\kappa = 1$ & $100\%$ & $5\%$ & $0\%$ & $100\%$ & $12\%$ & $0\%$ & $100\%$ & $14\%$ & $0\%$ & $100\%$ & $24\%$ & $0\%$\\ 
CBO & $100\%$ & $0\%$ & $0\%$ & $100\%$ & $0\%$ & $0\%$ & $100\%$ & $0\%$ & $0\%$ & $100\%$ & $0\%$ & $0\%$\\ 
\hline
\end{tabular}
\end{adjustbox}
\caption{Performance of polarized CBO for minimizing the multimodal Ackley function with \fbox{$d=2$} and three global minima: 
Averaging over $100$ independent runs of $1,000$ iterations we plot how many percent of the succeeded in detecting \emph{at least} $1,2$, or $3$ of the minima.
Note that, by definition, standard CBO ($\kappa=\infty$) can detect \emph{at most one} minimum.}
\label{tab:multimodal_ackley}
\end{table}

Even for $\kappa=\infty$, where the kernel has no influence anymore, the method works very well.
Note that this case does not correspond to standard CBO. Furthermore, \cref{tab:multimodal_ackley_d_10} shows that our polarized method can outperform standard CBO at finding at least one minimum, albeit at the cost of higher complexity.

We also test the cluster method in $d=30$ dimensions and the results can be found in \cref{tab:multimodal_ackley_MM_d_30}. We observe that it is harder to find multiple minima, however for $J=1600$ the method finds at least two minima in over $50\%$ of the runs for $\kappa\geq 1$. For smaller particle numbers the algorithm performs better for $\kappa \leq 1$, although the percentage of runs, where multiple minima are found is very low.

\begin{table}[htb]
\centering%
\setlength{\extrarowheight}{3pt}%
\begin{adjustbox}{width=\columnwidth,center}
\parbox{2mm}{\rotatebox[origin=c]{90}{\hspace{-30pt}Cluster \hspace{10pt} Polarized}}
\hspace{5pt}
\begin{tabular}{|l|c|c|c||c|c|c||c|c|c||c|c|c|}
\hline
                                &
  \multicolumn{3}{c||}{$J=50$}  &
  \multicolumn{3}{c||}{$J=100$} &
  \multicolumn{3}{c||}{$J=200$} &
  \multicolumn{3}{c|}{$J=400$}\\
\hline
$\#$ minima &
$\geq 1$ & $\geq 2$ & $\geq 3$ &
$\geq 1$ & $\geq 2$ & $\geq 3$ &
$\geq 1$ & $\geq 2$ & $\geq 3$ &
$\geq 1$ & $\geq 2$ & $\geq 3$ \\
\hline
$\kappa = 0.001$ & $5\%$ & $0\%$ & $0\%$ & $18\%$ & $0\%$ & $0\%$ & $26\%$ & $0\%$ & $0\%$ & $63\%$ & $1\%$ & $0\%$\\ 
$\kappa = 0.01$ & $26\%$ & $0\%$ & $0\%$ & $56\%$ & $0\%$ & $0\%$ & $80\%$ & $1\%$ & $0\%$ & $79\%$ & $3\%$ & $0\%$\\ 
$\kappa = 0.1$ & $36\%$ & $0\%$ & $0\%$ & $68\%$ & $0\%$ & $0\%$ & $73\%$ & $0\%$ & $0\%$ & $75\%$ & $0\%$ & $0\%$\\ 
CBO & $32\%$ & $0\%$ & $0\%$ & $55\%$ & $0\%$ & $0\%$ & $74\%$ & $0\%$ & $0\%$ & $72\%$ & $0\%$ & $0\%$\\ 
\hhline{|*{13}{=}|}
$\kappa = 10^{-7}$ & $26\%$ & $2\%$ & $0\%$ & $75\%$ & $26\%$ & $2\%$ & $98\%$ & $59\%$ & $13\%$ & $100\%$ & $89\%$ & $28\%$\\ 
$\kappa = 0.1$ & $11\%$ & $1\%$ & $0\%$ & $68\%$ & $13\%$ & $0\%$ & $98\%$ & $77\%$ & $19\%$ & $100\%$ & $96\%$ & $39\%$\\ 
$\kappa = \infty$ & $6\%$ & $0\%$ & $0\%$ & $65\%$ & $11\%$ & $0\%$ & $98\%$ & $73\%$ & $15\%$ & $100\%$ & $92\%$ & $41\%$\\ 
\hline
\end{tabular}
\end{adjustbox}
\caption{Performance of polarized and cluster CBO for minimizing the multimodal Ackley function with \fbox{$d=10$} and ceteris paribus.
The first method is slightly better than CBO but fails to detect more than one minimum in most cases.
The cluster variant works well for a wide range of parameters $\kappa$. \rev It manages to detect all minima in some runs, already for $J=200$ particles.\nc\
Note that $\kappa=\infty$ \emph{does not} correspond to CBO for this method.}
\label{tab:multimodal_ackley_d_10}
\end{table}

\begin{table}[htb]
\centering%
\setlength{\extrarowheight}{3pt}%
\begin{adjustbox}{width=\columnwidth,center}
\parbox{2mm}{\rotatebox[origin=c]{90}{\hspace{-25pt}Cluster}}
\hspace{5pt}
\begin{tabular}{|l|c|c|c||c|c|c||c|c|c||c|c|c|}
\hline
 &
  \multicolumn{3}{c||}{$J=200$} &
  \multicolumn{3}{c||}{$J=400$} &
  \multicolumn{3}{c||}{$J=800$} &
  \multicolumn{3}{c|}{$J=1600$}\\
\hline
$\#$ minima &
$\geq 1$ & $\geq 2$ & $\geq 3$ &
$\geq 1$ & $\geq 2$ & $\geq 3$ &
$\geq 1$ & $\geq 2$ & $\geq 3$ &
$\geq 1$ & $\geq 2$ & $\geq 3$ \\
\hline
$\kappa = 0.01$ & $15\%$ & $0\%$ & $0\%$ & $61\%$ & $5\%$ & $0\%$ & $92\%$ & $28\%$ & $1\%$ & $99\%$ & $45\%$ & $1\%$\\ 
$\kappa = 0.1$ & $9\%$ & $0\%$ & $0\%$ & $59\%$ & $6\%$ & $0\%$ & $94\%$ & $34\%$ & $2\%$ & $98\%$ & $49\%$ & $2\%$\\ 
$\kappa = 1$ & $4\%$ & $0\%$ & $0\%$ & $57\%$ & $5\%$ & $0\%$ & $88\%$ & $40\%$ & $0\%$ & $97\%$ & $57\%$ & $4\%$\\ 
$\kappa = 10$ & $5\%$ & $0\%$ & $0\%$ & $58\%$ & $7\%$ & $0\%$ & $89\%$ & $35\%$ & $1\%$ & $98\%$ & $58\%$ & $1\%$\\ 
$\kappa = 100$ & $4\%$ & $0\%$ & $0\%$ & $59\%$ & $7\%$ & $0\%$ & $89\%$ & $32\%$ & $0\%$ & $98\%$ & $57\%$ & $2\%$\\ 
$\kappa = \infty$ & $5\%$ & $0\%$ & $0\%$ & $58\%$ & $6\%$ & $0\%$ & $88\%$ & $33\%$ & $0\%$ & $98\%$ & $54\%$ & $2\%$\\ 
\hline
\end{tabular}
\end{adjustbox}
\caption{Performance of cluster CBO for minimizing the multimodal Ackley function with \fbox{$d=30$} and ceteris paribus.}
\label{tab:multimodal_ackley_MM_d_30}
\end{table}

\subsection{Multimodal sampling}

\begin{figure}[thb]
\begin{subfigure}{.24\textwidth}%
\includegraphics[width=\textwidth]{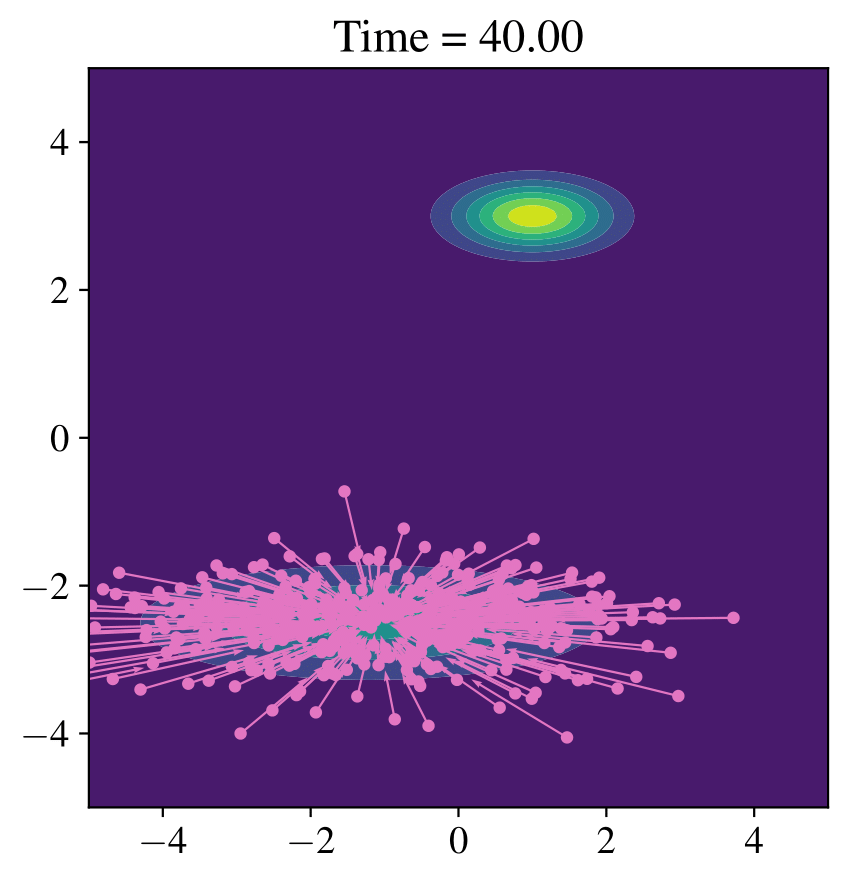}
\caption{Standard CBS.}
\end{subfigure}%
\hfill%
\begin{subfigure}{.24\textwidth}%
\includegraphics[width=\textwidth]{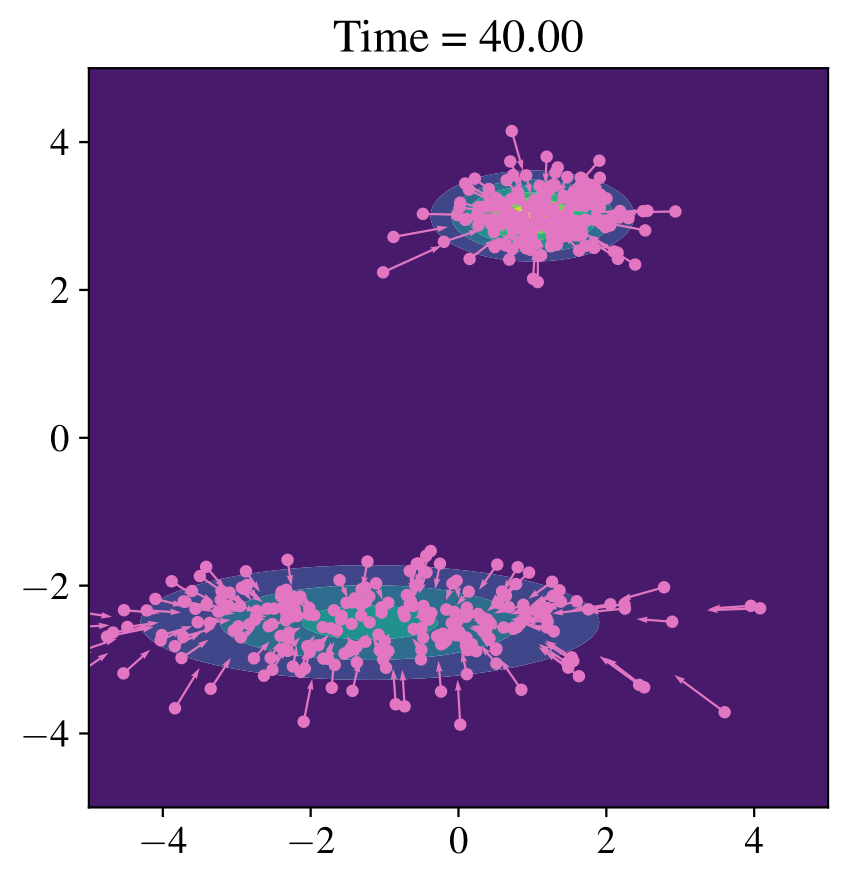}
\caption{$\kappa=0.8$.}
\end{subfigure}%
\hfill%
\begin{subfigure}{.24\textwidth}%
\includegraphics[width=\textwidth]{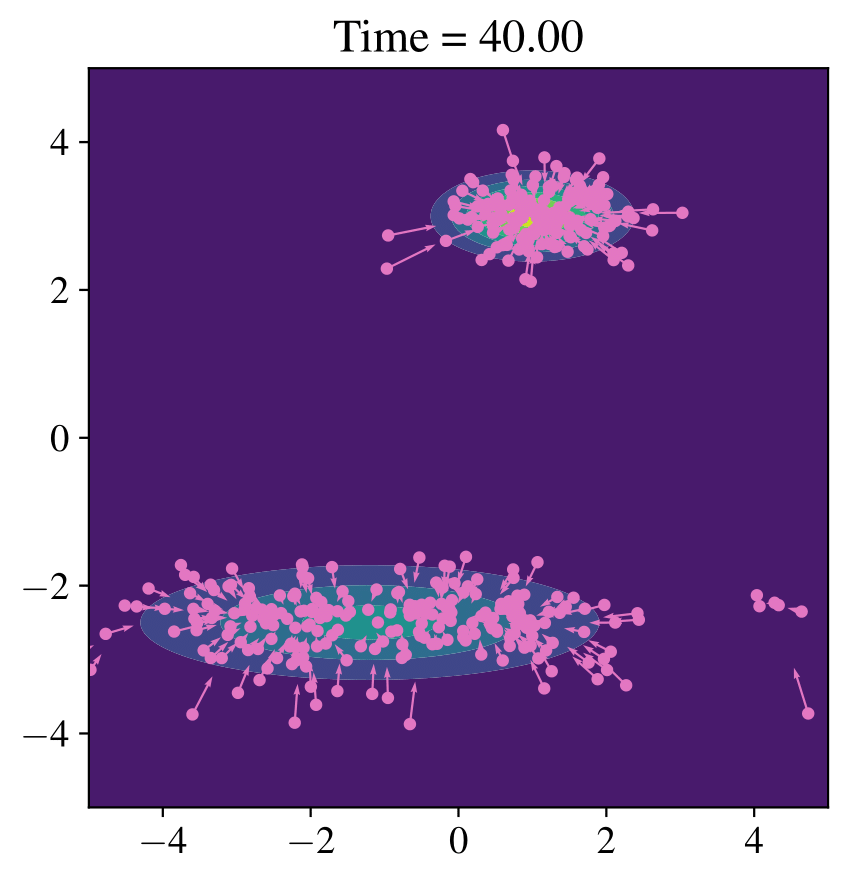}
\caption{$\kappa=0.6$.}
\end{subfigure}%
\hfill%
\begin{subfigure}{.24\textwidth}%
\includegraphics[width=\textwidth]{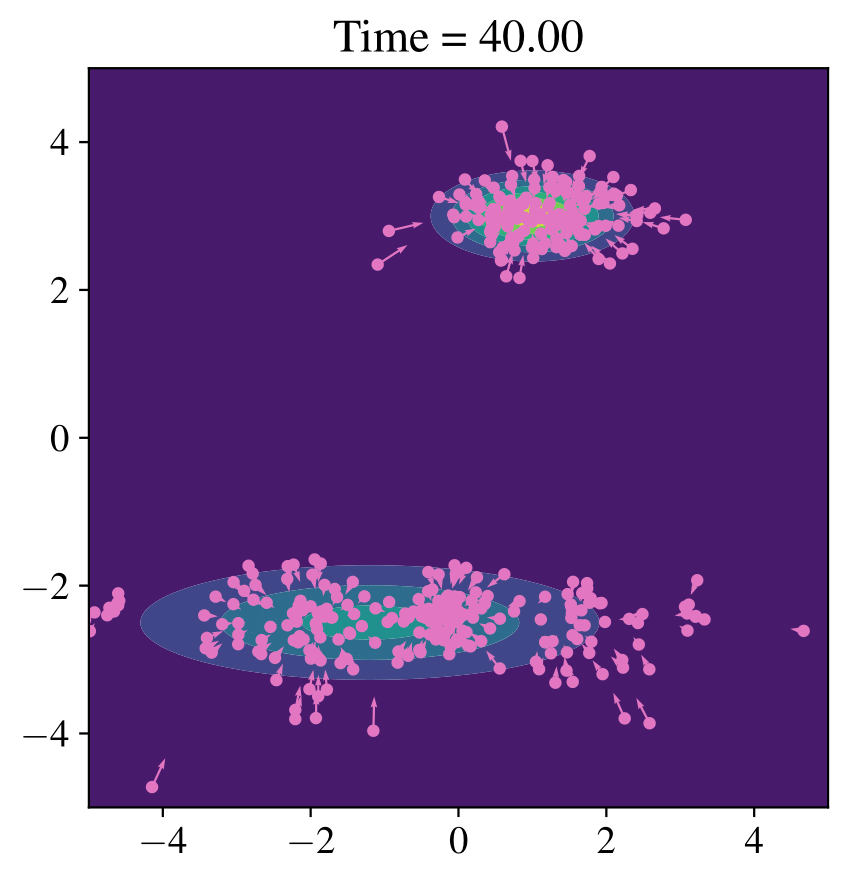}
\caption{$\kappa=0.4$.}\end{subfigure}%
\\
\begin{subfigure}{.24\textwidth}%
\includegraphics[width=\textwidth]{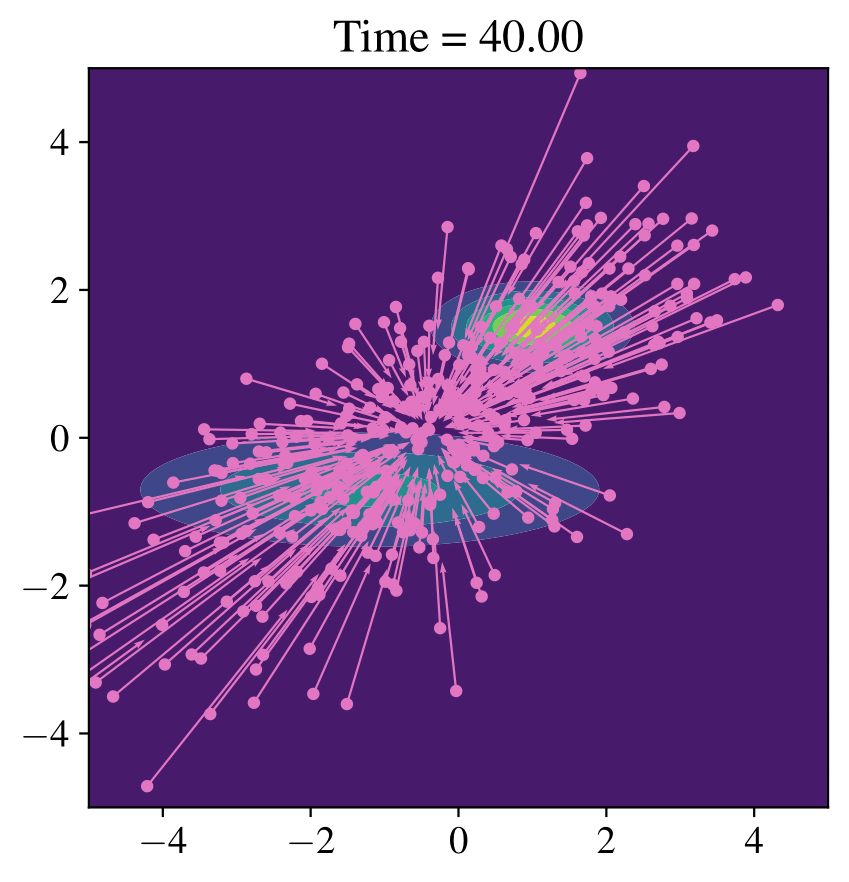}
\caption{Standard CBS.}
\end{subfigure}%
\hfill%
\begin{subfigure}{.24\textwidth}%
\includegraphics[width=\textwidth]{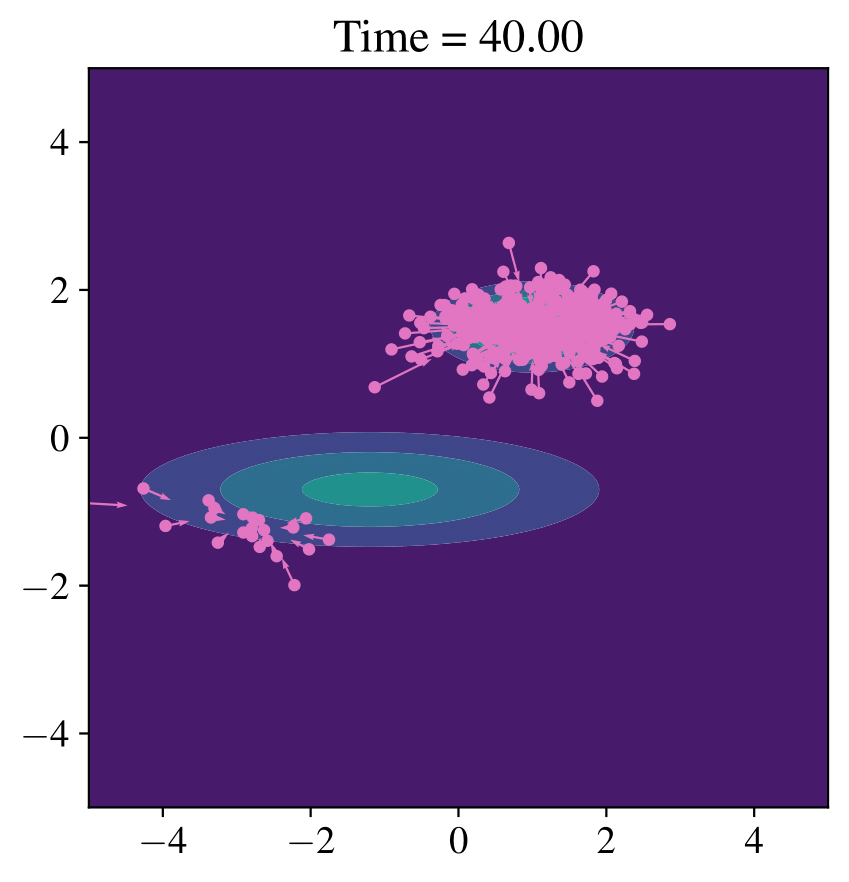}
\caption{$\kappa=0.8$.}
\end{subfigure}%
\hfill%
\begin{subfigure}{.24\textwidth}%
\includegraphics[width=\textwidth]{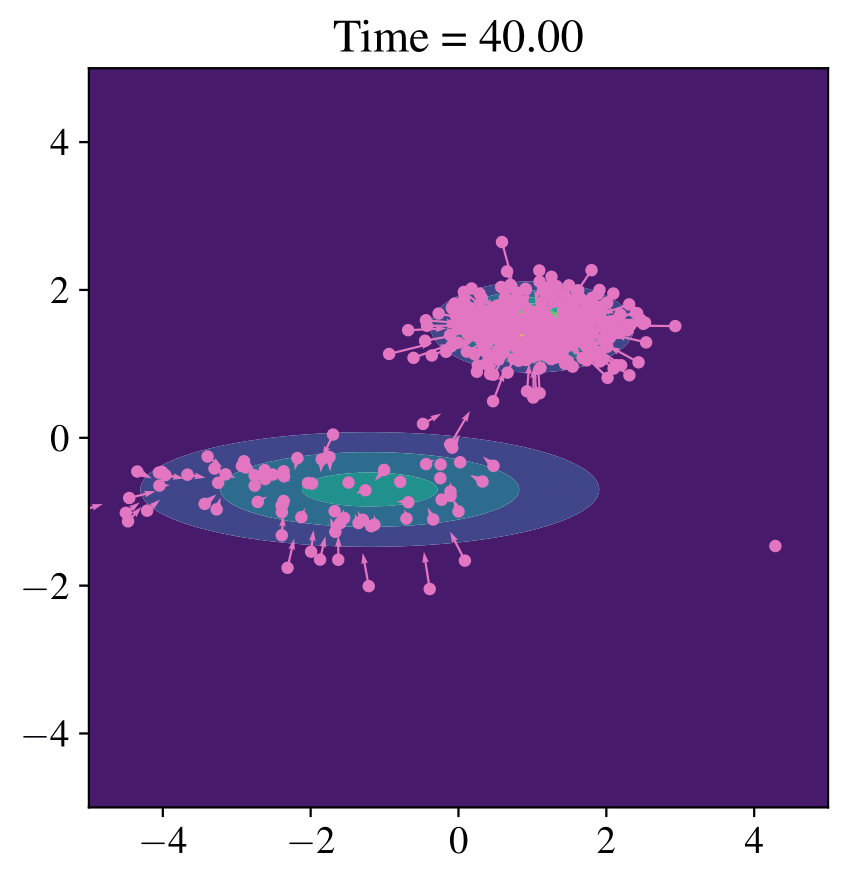}
\caption{$\kappa=0.6$.}
\end{subfigure}%
\hfill%
\begin{subfigure}{.24\textwidth}%
\includegraphics[width=\textwidth]{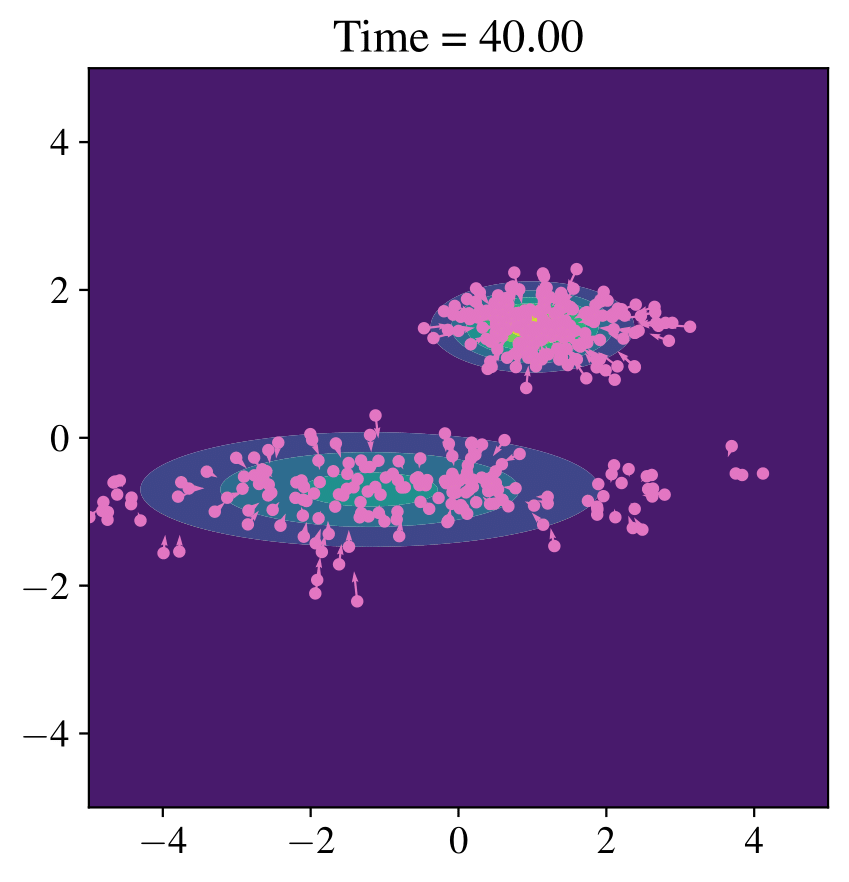}
\caption{$\kappa=0.4$.}
\end{subfigure}%
\caption{Dynamics of standard CBS and our polarized version, sampling from a mixture of Gaussians (top row: far apart, bottom row: closer together). 
The points mark particle locations, the arrows the drift field towards the weighted mean(s).
\label{fig:MixOfGaussians_CBS}}
\end{figure}

In this section we consider the task of sampling from a bimodal mixture of Gaussians.
\begin{align*}
\resizebox{\textwidth}{!}{$%
\exp(-V(x)) := 
\exp\left(-(x_1-a_1)^2 - \frac{(x_2-a_2)^2}{0.2}\right)
+
\frac{1}{2}\exp\left( -\frac{(x_1-b_1)^2}{8} - \frac{(x_2-b_2)^2}{0.5}\right),$}
\end{align*}
where the tuples $a=(a_1,a_2)$ and $b=(b_1,b_2)$ determine the centers of the clusters.
In \cref{fig:MixOfGaussians_CBS} we plot the result of standard CBS and our polarized variant using Gaussian kernels with three different standard deviations $\kappa\in\{0.4,0.6,0.8\}$.
For both methods we choose $\beta=1$.
Note that, at least for convex potentials $V$ with bounded Hessian, standard CBS is known to exhibit a Gaussian steady state. 
Being designed as a method for unimodal sampling, there is not much hope CBS can work in this multimodal situation.
Still we include CBS results for comparison. 

In contrast, our polarized modification of CBS manages to isolate the two modes. 
Note that in \cref{prop:unbiased} we proved that polarized CBS with a Gaussian kernel is unbiased when when the target measure is a Gaussian. 
We do not expect this to be true for target measures which are a mixture of Gaussians. 
Still, our results in the first row of \cref{fig:MixOfGaussians_CBS} show that, if the two Gaussians are sufficiently far apart, our method (second to fourth column) seems close to being unbiased.
Standard CBS (left column) can, by design, find at most one mode and here it successfully detects the lower Gaussian.
Note that we use the same number of particles for both algorithms, namely $J=400$, which is why the bottom mode for CBS looks more densely sampled.

The situation is different when the two modes are closer together, as shown in the bottom row of \cref{fig:MixOfGaussians_CBS}.
Here, standard CBS fails to detect even one mode whereas our polarized version does achieve that. 
However, as expected in this multimodal situation the result is not a perfect sample but appears to be biased.
Furthermore, if the clusters are close, the sensitivity with respect to the choice of the kernel width $\kappa$ is larger and $\kappa$ has to be chosen sufficiently small in order to generate enough samples for the lower cluster.

\subsection{Non-Gaussian sampling}

\begin{figure}[htb]
    \centering
    \begin{subfigure}{.4\textwidth}%
    \includegraphics[width=\textwidth]{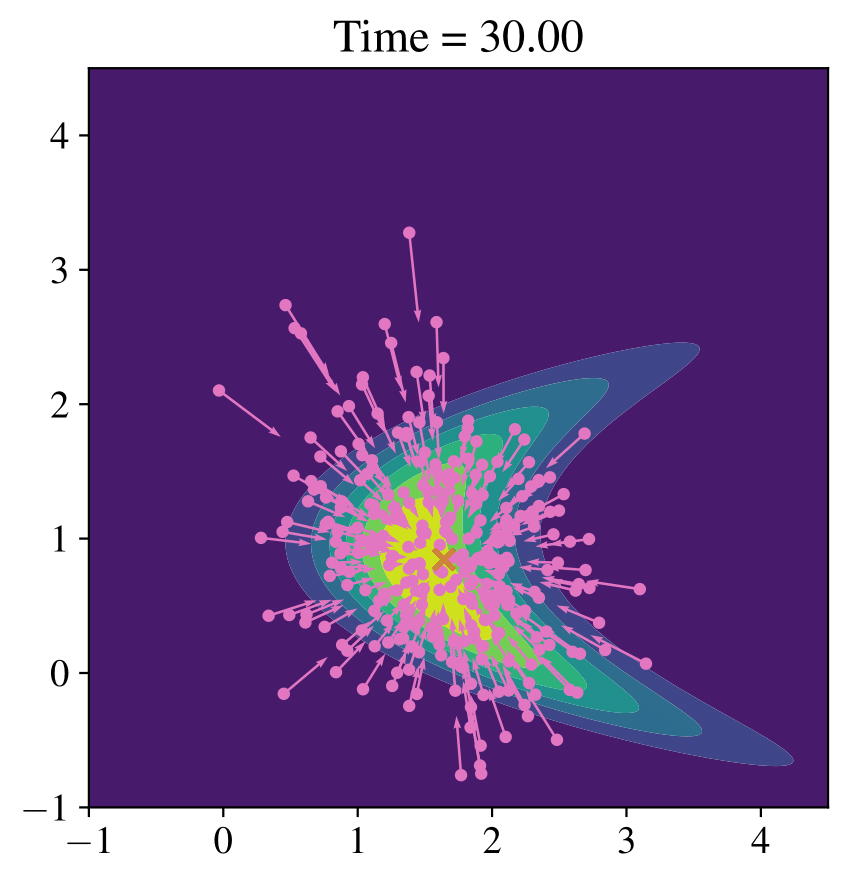}
    \caption{Standard CBS.}
    \end{subfigure}%
    \hspace*{2em}%
    \begin{subfigure}{.4\textwidth}%
    \includegraphics[width=\textwidth]{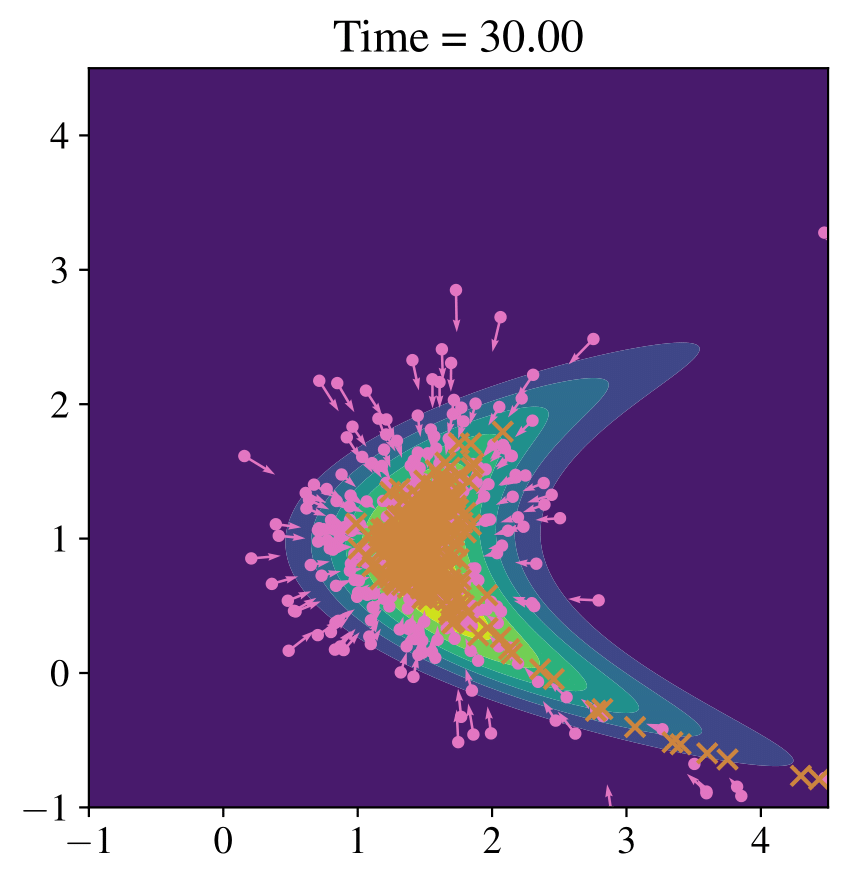}
    \caption{Polarized CBS.}
    \end{subfigure}%
    \caption{Standard and polarized CBS for sampling from a non-Gaussian distribution. The orange crosses mark the position of the weighted means.}
    \label{fig:banana}
\end{figure}

Our final experiment deals with sampling from a non-Gaussian distribution.
Again, we use $\beta=1$.
According to the theoretical analysis in \cite{carrillo2022consensus_sampling} standard CBS is not expected to correctly sample from such a distribution but rather from a Gaussian approximation.
While we do not claim that our polarized method generates an exact sample, our result in \cref{fig:banana} show that polarized CBS approximates the non-Gaussian distribution much better than standard CBS. 
This is due to the fact that the weighted means, indicated as orange crosses, do not collapse to a single point but concentrate in the region of high probability mass.

\section{Conclusion and outlook}

In this article we presented a polarized version of consensus-based optimization and sampling dynamics for objectives with multiple global minima or modes.
For this we localized the dynamics such that every particle is primarily influenced by close-by particles. 
We proved that in the case of sampling from a Gaussian distribution this does not introduce a bias. 
We also suggested a cluster-based version of our polarized dynamics which is computationally more efficient.
Our extensive numerical experiments suggested a large potential of our method for detecting multiple global minima or modes, improving over standard consensus-based methods.

There is a lot of room for future work regarding well-definedness of the Fokker--Planck equation derived above, and stability and convergence of both the mean-field and particle system to consensus \rev using less restrictive assumptions than the ones used in this paper.
We are convinced that the Lyapunov function $\L[\rho_t]$ which we studied in \cref{sec:analysis} will be very helpful in this endeavour.
\nc
Future work will also focus on further numerical improvements of our method, in particular, incorporating a batching strategy similar to the one from \cite{carrillo2021consensus_ML} to further improve performance in high-dimensional optimization.
Finally, a long-term goal will be to find multimodal sampling methods which are provably consistent for multimodal and non-Gaussian distributions.
Given that even gradient-free sampling from unimodal non-Gaussian distributions is still relatively poorly understood, with \cite{schillings2022ensemble} being a promising approach, this will be a challenging task for future work.

\section*{Acknowledgments}

Part of this work was also done while LB and TR were in residence at Institut Mittag-Leffler in Djursholm, Sweden during the semester on \textit{Geometric Aspects of Nonlinear Partial Differential Equations} in 2022, supported by the Swedish Research Council under grant no. 2016-06596.
LB is very thankful to Massimo Fornasier and Oliver Tse for insightful discussions on the analytical aspects of this paper.
Part of this work was done while LB was affiliated with the University of Bonn and the Technical University of Berlin.
PW acknowledges support from MATH+ project EF1-19: Machine Learning Enhanced Filtering Methods for Inverse Problems, funded by the Deutsche Forschungsgemeinschaft (DFG, German Research Foundation) under Germany's Excellence Strategy – The Berlin Mathematics Research Center MATH+ (EXC-2046/1, project ID: 390685689). 
TR acknowledges support by the German Ministry of Science and Technology (BMBF) under grant agreement No. 05M2020 (DELETO) and was supported by the European Unions Horizon 2020 research and innovation programme under the Marie Sk{\l}odowska-Curie grant agreement No 777826 (NoMADS). Most of this study was carried out while TR was affiliated with the Friedrich-Alexander-Universität Erlangen-Nürnberg.

\printbibliography

@book{ackley2012,
  title={A Connectionist Machine for Genetic Hillclimbing},
  author={Ackley, D.},
  isbn={9781461319979},
  lccn={87013536},
  series={The Springer International Series in Engineering and Computer Science},
  year={2012},
  address = {New York, NY},
  publisher={Springer US}
}

@article{carrillo2018analytical,
  title={An analytical framework for consensus-based global optimization method},
  author={Carrillo, Jos{\'e} A and Choi, Young-Pil and Totzeck, Claudia and Tse, Oliver},
  journal={Mathematical Models and Methods in Applied Sciences},
  volume={28},
  number={06},
  pages={1037--1066},
  year={2018},
  publisher={World Scientific}
}

@article{pinnau2017consensus,
  title={A consensus-based model for global optimization and its mean-field limit},
  author={Pinnau, Ren{\'e} and Totzeck, Claudia and Tse, Oliver and Martin, Stephan},
  journal={Mathematical Models and Methods in Applied Sciences},
  volume={27},
  number={01},
  pages={183--204},
  year={2017},
  publisher={World Scientific}
}

@article{carrillo2022consensus_sampling,
  title={Consensus-based sampling},
  author={Carrillo, Jos\'e A and Hoffmann, Franca and Stuart, Andrew M and Vaes, Urbain},
  journal={Studies in Applied Mathematics},
  volume={148},
  number={3},
  pages={1069--1140},
  year={2022},
  publisher={Wiley Online Library}
}

@article{ha2020convergence,
  title={Convergence of a first-order consensus-based global optimization algorithm},
  author={Ha, Seung-Yeal and Jin, Shi and Kim, Doheon},
  journal={Mathematical Models and Methods in Applied Sciences},
  volume={30},
  number={12},
  pages={2417--2444},
  year={2020},
  publisher={World Scientific}
}

@article{carrillo2021consensus_ML,
  title={A consensus-based global optimization method for high dimensional machine learning problems},
  author={Carrillo, Jos{\'e} A and Jin, Shi and Li, Lei and Zhu, Yuhua},
  journal={ESAIM: Control, Optimisation and Calculus of Variations},
  volume={27},
  pages={S5},
  year={2021},
  publisher={EDP Sciences}
}

@incollection{totzeck2022trends,
  title={Trends in consensus-based optimization},
  author={Totzeck, Claudia},
  booktitle={Active Particles, Volume 3},
  pages={201--226},
  year={2022},
  address = {Cham},
  publisher={Springer}
}

@article{fornasier2021consensus_conv,
  title={Consensus-based optimization methods converge globally in mean-field law},
  author={Fornasier, Massimo and Klock, Timo and Riedl, Konstantin},
  journal={arXiv preprint arXiv:2103.15130},
  year={2021}
}

@article{fornasier2020consensus_hyper,
  title={Consensus-based optimization on hypersurfaces: Well-posedness and mean-field limit},
  author={Fornasier, Massimo and Huang, Hui and Pareschi, Lorenzo and S{\"u}nnen, Philippe},
  journal={Mathematical Models and Methods in Applied Sciences},
  volume={30},
  number={14},
  pages={2725--2751},
  year={2020},
  publisher={World Scientific}
}

@article{fornasier2021consensus_ML,
  title={Consensus-Based Optimization on the Sphere: Convergence to Global Minimizers and Machine Learning.},
  author={Fornasier, Massimo and Pareschi, Lorenzo and Huang, Hui and S{\"u}nnen, Philippe},
  journal={J. Mach. Learn. Res.},
  volume={22},
  number={237},
  pages={1--55},
  year={2021}
}

@article{carrillo2021consensus_constrained,
  title={Consensus-based optimization and ensemble Kalman inversion for global optimization problems with constraints},
  author={Carrillo, Jos{\'e} A and Totzeck, Claudia and Vaes, Urbain},
  journal={arXiv preprint arXiv:2111.02970},
  year={2021}
}

@article{petersen2008matrix,
  title={The matrix cookbook},
  author={Petersen, Kaare Brandt and Pedersen, Michael Syskind and others},
  journal={Technical University of Denmark},
  volume={7},
  number={15},
  pages={510},
  year={2008}
}

@article{rastrigin1974systems,
  title={Systems of extremal control},
  author={Rastrigin, Leonard Andreevi{\v{c}}},
  journal={Nauka},
  year={1974}
}

@article{hegselmann2002opinion,
  title={Opinion dynamics and bounded confidence models, analysis, and simulation},
  author={Hegselmann, Rainer and Krause, Ulrich},
  journal={Journal of artificial societies and social simulation},
  volume={5},
  number={3},
  year={2002}
}

@article{gomez2012bounded,
  title={The bounded confidence model of opinion dynamics},
  author={G{\'o}mez-Serrano, Javier and Graham, Carl and Le Boudec, Jean-Yves},
  journal={Mathematical Models and Methods in Applied Sciences},
  volume={22},
  number={02},
  pages={1150007},
  year={2012},
  publisher={World Scientific}
}

@article{fortunato2005vector,
  title={Vector opinion dynamics in a bounded confidence consensus model},
  author={Fortunato, Santo and Latora, Vito and Pluchino, Alessandro and Rapisarda, Andrea},
  journal={International Journal of Modern Physics C},
  volume={16},
  number={10},
  pages={1535--1551},
  year={2005},
  publisher={World Scientific}
}

@article{deffuant2000mixing,
  title={Mixing beliefs among interacting agents},
  author={Deffuant, Guillaume and Neau, David and Amblard, Frederic and Weisbuch, G{\'e}rard},
  journal={Advances in Complex Systems},
  volume={3},
  number={01n04},
  pages={87--98},
  year={2000},
  publisher={World Scientific}
}

@article{fukunaga1975estimation,
  title={The estimation of the gradient of a density function, with applications in pattern recognition},
  author={Fukunaga, Keinosuke and Hostetler, Larry},
  journal={IEEE Transactions on information theory},
  volume={21},
  number={1},
  pages={32--40},
  year={1975},
  publisher={IEEE}
}

@article{schnell1964methode,
  title={Eine Methode zur Auffindung von Gruppen},
  author={Schnell, P},
  journal={Biometrische Zeitschrift},
  volume={6},
  number={1},
  pages={47--48},
  year={1964}
}

@article{burger2022kinetic,
    title = {Kinetic equations for processes on co-evolving networks},
    journal = {Kinetic and Related Models},
    volume = {15},
    number = {2},
    pages = {187-212},
    year = {2022},
    doi = {10.3934/krm.2021051},
    author = {Martin Burger}
}

@article{burger2021network,
  title={Network structured kinetic models of social interactions},
  author={Burger, Martin},
  journal={Vietnam Journal of Mathematics},
  volume={49},
  number={3},
  pages={937--956},
  year={2021},
  publisher={Springer}
}

@inproceedings{kennedy1995particle,
  author={Kennedy, J. and Eberhart, R.},
  booktitle={Proceedings of ICNN'95 - International Conference on Neural Networks},
  title={Particle swarm optimization}, 
  year={1995},
  volume={4},
  pages={1942--1948}
}

@inproceedings{reynolds1987flocks,
  title={Flocks, herds and schools: A distributed behavioral model},
  author={Reynolds, Craig W},
  booktitle={Proceedings of the 14th annual conference on Computer graphics and interactive techniques},
  pages={25--34},
  year={1987}
}

@article{totzeck2020personal,
    title = {Consensus-based global optimization with personal best},
    journal = {Mathematical Biosciences and Engineering},
    volume = {17},
    number = {5},
    pages = {6026-6044},
    year = {2020},
    issn = {1551-0018},
    doi = {10.3934/mbe.2020320},
    author = {Claudia Totzeck and  Marie-Therese Wolfram}
}

@book{conn2009introduction,
  title={Introduction to derivative-free optimization},
  author={Conn, Andrew R and Scheinberg, Katya and Vicente, Luis N},
  year={2009},
  publisher={SIAM},
  address={Philadelphia, PA}
}

@article{schillings2022ensemble,
  title={Ensemble-based gradient inference for particle methods in optimization and sampling},
  author={Schillings, Claudia and Totzeck, Claudia and Wacker, Philipp},
  journal={arXiv preprint arXiv:2209.15420},
  year={2022}
}

\end{document}